\documentclass[12pt,a4paper,reqno]{amsart}

\usepackage[latin1]{inputenc}
\usepackage[T1]{fontenc}
\usepackage{mathtools}
\usepackage[backend=biber,
sorting=nyt]{biblatex}
\AtBeginBibliography{\footnotesize}
\renewbibmacro{in:}{}
\usepackage{amsfonts}
\usepackage{amsmath}
\usepackage{amssymb}
\usepackage{eurosym}
\usepackage{mathrsfs}
\usepackage{palatino}
\usepackage{xcolor}
\colorlet{linkequation}{blue}
\usepackage[colorlinks=true, citecolor=blue, linkcolor=violet]{hyperref}
\allowdisplaybreaks

\DeclareMathOperator{\supp}{supp}

\newcommand{\wt}{\widetilde}
\newcommand{\cg}{\Gamma}
\renewcommand{\L}{\mathcal{L}}
\newcommand{\F}{\mathcal{F}}
\renewcommand{\k}{\kappa}
\newcommand{\g}{\mathfrak{g}}
\newcommand{\vf}{\mathfrak{v}}
\newcommand{\inn}{{\it i}}
\newcommand{\out}{{\it o}}
\newcommand{\Rinn}{R^{\inn}}
\newcommand{\Rout}{R^{\out}}
\newcommand{\cD}{{\mathcal D}}
\newcommand{\cH}{{\mathcal H}}
\renewcommand{\l}{\lambda}
\newcommand{\D}{\mathcal{D}}
\newcommand{\R}{\mathbb{R}}
\newcommand{\C}{\mathbb{C}}
\renewcommand{\a}{\alpha}
\newcommand{\G}{\mathbb{G}}
\newcommand{\diam}{\operatorname{diam}}
\newcommand{\h}{\mathcal{H}_{\infty}}

\numberwithin{equation}{section}

\usepackage{amsthm}

\newtheorem{thm}{Theorem}[section]
\newtheorem{prop}[thm]{Proposition}
\newtheorem{cor}[thm]{Corollary}
\newtheorem{lem}[thm]{Lemma}

\theoremstyle{definition}
\newtheorem{defn}[thm]{Definition}
\newtheorem{exmp}[thm]{Example}
\newtheorem{rem}[thm]{Remark}

\newcommand{\restrict}{\begin{picture}(12,12)
                        \put(2,0){\line(1,0){8}}
                        \put(2,0){\line(0,1){8}}
                       \end{picture}}

\renewcommand*{\star}[2][-3mu]{\ensuremath{\mskip1mu\prescript{\smash{\mathrm *\mkern#1}}{}{\mathstrut#2}}}

\usepackage[margin=2.0cm]{geometry}

\title{Capacities Characterizing Removable Sets for Various Function Spaces in Carnot Groups}
\author{Zack Boone}
\address{Department of Mathematics, University of Connecticut}
\email{zackary.boone@uconn.edu}

\begin{document}

\begin{abstract}
We study removable sets for the Campanato, H\"{o}lder continuous, $L^p_{\text{loc}}$, and Lipschitz functions in Carnot groups. In the former three cases, we characterize removability through the use of capacities with respect to any left-invariant linear differential operator $\L$ for which $\L$ and $\L^t$ are hypoelliptic and satisfy a homogeneity condition, while in the latter case we characterize Lipschitz functions with respect to the sub-Laplacian.
\end{abstract}

\maketitle

\section{Introduction}
For a given space of distributions $\mathcal{F}$ and a differential operator $\L$, a set $K$ is removable for $\L$-solutions in $\mathcal{F}$ if for any domain $\Omega$ and any element $f \in \mathcal{F}(\Omega)$ satisfying $\L f = 0$ in $\Omega \backslash K$ also satisfies $\L f = 0$ in $\Omega$. The notion of removable sets first arose in 1888 in the complex plane when Painlev\'e considered when a bounded analytic function on $\Omega \backslash K$ has an analytic extension to $\Omega$. 

One can think of a removable set for $\L$-solutions in $\F$ as a set which contains no "information" with respect to $\mathcal{F}$ and $\L$. In the context of bounded analytic functions, this was made more precise by Ahlfors in 1947 \cite{Ah} when he defined a set function called analytic capacity, denoted by $\gamma$, and showed that a compact set $K \subset \C$ is removable for bounded analytic functions if and only if $\gamma(K) = 0$. Hence, the "null sets" of this set function are precisely the removable sets for bounded analytic functions in $\C$. Ahlfors was motivated by the Painlev\'{e} problem, which asks to give a geometric condition that characterizes removable sets for bounded analytic functions. Analytic capacity is not a geometric condition; however, the study of analytic capacity helped to give a full solution to Painlev\'{e}'s problem which was achieved by Tolsa \cite{To}. Tolsa's work \cite{To} was built upon the work of many authors including Ahlfors, Denjoy, Garnett, Calder\'{o}n, David, Mattila, Mel'nikov, Verdera, Nazarov, Treil, and Volberg just to name a few. See Tolsa's article \cite{Toref} written for the proceedings of the 2006 ICM for more history and results.

The study of removable sets for general differential operators and function spaces has been studied by a wide range of authors such as Carleson, Kr\'{a}l, and Harvey-Polking. In 1963, Carleson \cite{Car} gave a removability characterization of H\"{o}lder continuous harmonic functions functions in $\R^n$. Harvey-Polking's results in \cite{HP} give sufficient conditions for when a set is removable for a large class of linear partial differential equations. They achieved this level of generality by proving the existence of a partition of unity, now called the Harvey-Polking partition of unity \cite[Lemma 3.1]{HP}, which incorporates derivative information. This partition of unity is fundamental to the study of removable sets and has also been applied to problems of uniform approximation \cite{P}. The results obtained by  Kr\'{a}l's \cite{Kr}  were quite general, characterizing removable sets for semielliptic partial differential operators for function spaces such as BMO, VMO, Campanato, and H\"{o}lder continuous functions. 

The Lipschitz harmonic capacity was introduced by Paramonov \cite[(2.2)]{P}, motivated by problems in uniform approximation. Mattila-Paramonov \cite{MP} also showed that the Lipschitz harmonic capacity characterizes bounded removable sets for Lipschitz harmonic functions \cite[Proposition 2.2]{MP}. In a follow up paper to \cite{HP}, Harvey-Polking in \cite{HP2} developed capacities in a very general form. Using these capacities they were able to characterize removable sets for large classes of differential operators and function spaces.

Removable sets have also been studied in the sub-Riemannian setting, for example \cite{CMat}, \cite{CMT}, \cite{CT}, \cite{CFO}, \cite{CLZ}. Capacities on the other hand have not been studied in the sub-Riemannian setting, which we do in this paper. We will adapt the capacities developed by Harvey-Polking and Paramonov in \cite{HP2} and \cite{P} respectively to the setting of Carnot groups, which we denote by $\G$ and the background of which will be given in Section \ref{section-2}.

Let $\L$ denote the sub-Laplacian and $\nabla_{\G}$ the horizontal gradient. To avoid technicalities, we will work with the sub-Laplacian and horizontal gradient which is generated from the Jacobian basis. See Section \ref{section-2} for more details. We define the Carnot Lipschitz harmonic capacity, denoted by $\k$, of a bounded set $K$ as \begin{align*}
    \k(K) := \sup \{ |\langle \L f, 1 \rangle| : f \in \text{Lip}_{\text{loc}}, \supp \L f \subset K, \| \nabla_{\G} f\| \leq 1\}. 
\end{align*}
Above, $\supp \L f$ refers to the distributional support of $\L f$. The proof that $\k$ characterizes removable sets for Lipschitz harmonic functions in Euclidean space from Mattila-Paramonov relies on \cite[Lemma 4.2]{P}. There is an issue when one tries to adapt this proof to the Carnot setting, which is that the differential operators that we usually deal with in $\G$ do not commute, which is indeed used in the proof of \cite[Lemma 4.2]{P}. The reader will see later that for $f$ in some distribution class $\F$, when one is trying to characterize removability it becomes natural to try to show $\psi \L f \ast \cg \in \F$ where $\psi \in C_0^{\infty}$ and $\cg$ is the fundamental solution to $\L$. We are able to show this when $\F = \text{Lip}_{\text{loc}}$ by getting a distributional formula for $\psi \L f \ast \cg$ that involves sums of functions which have horizontal derivatives on $f$, then proving that each of the functions in this sum satisfy the conditions of $\k$. Using this idea, we are still able to prove, like in the Euclidean setting from Mattila-Paramonov, the  the characterization of removable (Carnot) Lipschitz harmonic functions through the null sets of the (Carnot) Lipschitz harmonic capacity $\k$. 
\begin{thm}\label{lip-cap-char}
    Let $K \subset \G$ be bounded and $\L$ the sub-Laplacian. Then $K$ is removable for Lipschitz $\L$-solutions if and only if $\k(K) = 0$.
\end{thm}
 We heavily rely on the fact that the sub-Laplacian is generated from horizontal vector fields and that the sub-Laplacian is $2$-homogeneous. We mention that Theorem \ref{lip-cap-char} also holds when $\L$ is a $2$-homogeneous linear differential operator generated by the horizontal vector fields from the Jacobian basis and for which $\L$ and $\L^t$ (the transpose of $\L$) are both hypoelliptic. The condition $\L$ and $\L^t$ being hypoelliptic guarantees the existence of a fundamental solution with certain nice homogeneity properties. See \cite[Theorem 2.1]{Fo}.

Characterizing removable sets for Campanato functions and H\"{o}lder continuous functions through any left-invariant homogeneous linear differential operator $\L$ for which both $\L$ and $\L^t$ are hypoelliptic, will actually come as consequences of our other main results. The task of characterizing removability for Campanato and H\"{o}lder continuous functions turns out to be more delicate than showing Theorem \ref{lip-cap-char}, which we now explain why. As $\L$ and $\L^t$ are hypoelliptic, we have the existence of a fundamental solution $\Gamma$ to $\L$. In contrast to the local Lipschitz case, we are not able to directly show that $\psi \L f \ast \cg \in \F$ where $\F$ is either the Campanato or H\"{o}lder space and $f \in \F$. In the Campanato case, the reason for this is not having access to a commutative group law and commutative differential operators. For the H\"{o}lder case, it is because the functions $f$ need not be differentiable almost everywhere. So for these more difficult cases, we will take a different strategy and the usefulness of the (Carnot) Harvey-Polking partition of unity in $\G$, proven by Chousionis and Tyson in \cite[Lemma 3.8]{CT} will take effect. Sections \ref{section-4} and \ref{section-6} highlight this alternative approach where we will show that certain capacities are comparable to Hausdorff content and get a removability characterization as a consequence. Let $\k^{\varrho}$ and $\k^{\delta}$ denote the Campanato and H\"{o}lder capacities respectively. Our two other main results are the following: 

\begin{thm}\label{intro-cap}
    Let $\L$ be a left-invariant differential operator homogeneous of degree $\l \in [1, Q)$ such that both $\L$ and $\L^t$ are hypoelliptic. For a compact set $K \subset \G$ and $\varrho \in [\l, Q]$ we have, \begin{align*}
        \h^{\varrho -\l}(K) \lesssim \k^{\varrho}(K) \lesssim \h^{\varrho -\l}(K).
    \end{align*}
\end{thm}

\begin{thm}\label{intro-hold}
 Let $\L$ be a left-invariant differential operator homogeneous of degree $\l \in [1, Q)$ such that both $\L$ and $\L^t$ are hypoelliptic. For a compact set $K \subset \G$ and $\delta \in (0,1)$ we have, \begin{align*}
        \h^{Q -\l+\delta}(K) \lesssim \k^{\delta}(K) \lesssim \h^{Q -\l+\delta}(K).
    \end{align*}   
\end{thm}
The idea for proving Theorems \ref{intro-cap} and \ref{intro-hold} was inspired by a result of Verdera \cite{Ver}.

We also consider removable sets for $L_{\text{loc}}^p$, $p \in [1, \infty)$, which we can characterize in a similar way to showing Theorem \ref{lip-cap-char}. 

This paper is organized as follows. Section \ref{section-2} will give the necessary background about Carnot groups and differential operators. Section \ref{section-3} describes a generalization of the Harvey-Polking lemma and the version of Hausdorff measure that is more conducive to our calculations. The rest of the sections are devoted to characterizing removable sets through capacities for various function spaces.

\begin{subsection}*{Acknowledgments}
I want to thank my advisor Vasilis Chousionis for introducing this topic to me and providing many helpful comments during the process of writing this paper. I also thank Surath Fernando for helpful advice and comments during this course of this project. I was partially supported by the NSF grant 2247117.
\end{subsection}

\vskip2em

\section{Definitions and Preliminaries}\label{section-2}
\begin{subsection}{Carnot groups}
A \textit{Carnot group}, denoted $\G$, is a connected, simply connected, and nilpotent Lie group whose Lie algebra admits a stratification of the form \begin{align*}
    \g = \vf_1 \oplus \cdots \oplus \vf_s, \quad [\vf_1, \vf_i] = \vf_{i+1} \text{ for }i = 1, \cdots, s-1, \quad [\vf_1, \vf_s]= \{0\}
\end{align*}
and $\vf_1, \cdots, \vf_s$ are non-zero subspaces of $\g$. We call the integer $s \geq 1$ the \textit{step} of $\G$. Denote the group law of $\G$ by $\cdot$ the identity element $\G$ by $0$. We write $\G = (\R^N, \cdot)$. Notice that the condition $[\vf_1, \vf_i] = \vf_{i+1}$ for $i = 1, \cdots, s-1$ means that $\vf_1$ generates $\g$ through the Lie algebra. We call $\vf_1$ the \textit{horizontal layer} of the Lie algebra and the elements of $\vf_1$ are called \textit{horizontal vector fields}. Let,\begin{equation}\label{dim-of-first-layer}
m_1 := \dim(\vf_1).
\end{equation} 

The elements of $\g$ are left-invariant vector fields and from the conditions of $\G$, the exponential map $\exp : \g \rightarrow \G$ is a global diffeomorphism. Elements in $\g$ then act on smooth functions $f$ by, 
\begin{equation}\label{defn-of-left-inv-vf}
(X f)(x) = Xf(x) = \frac{d}{dt}\bigg|_{t=0} \left( f(x\cdot \exp(tX) ) \right).  
\end{equation} 
Given any left-invariant vector field $X \in \g$ we can also define a right-invariant vector field $\wt{X}$ by, 
\begin{equation}\label{defn-of-right-inv-vf}
(\wt{X}f)(x) = \wt{X}f(x) = \frac{d}{dt}\bigg|_{t=0} \left( f(\exp(tX)\cdot x) \right).
\end{equation}
Notice that equations (\ref{defn-of-left-inv-vf}) and (\ref{defn-of-right-inv-vf}) are only defined once the functions $Xf$ and $\wt{X}f$ have been found. Once they have been found, their definition becomes a derivative in the $t$-variable, and the $x$-variable is then treated as a constant.

Let $X$ be a vector field in $\R^N$ of the form $X h(x)= \sum_{i=1}^N a_i(x)\partial_j h(x)$ where $h$ is a smooth enough function.. Let \begin{align*}
    XI(x) = \begin{pmatrix}
        a_1(x) \\ \vdots \\ a_N(x)
    \end{pmatrix}
\end{align*}
and notice \begin{equation}\label{XI-notation}
    Xh(x) = \langle \nabla h(x), XI(x)\rangle
\end{equation} where $\nabla$ is the Euclidean gradient.

Fix an inner product on $\vf_1$. From here on out, we will take a basis to $\vf_1$ to consist of a certain kind of orthonormal basis of elements, called the \emph{Jacobian basis} of $\vf_1$. We describe its properties now.
\begin{prop}\cite[Proposition 1.2.16]{BLU}
Let $j \in \{1, \cdots, N\}$ be given. There exists one and only one vector field, denoted by $X_j$, characterized by any one of the following equivalent conditions:
\begin{enumerate}
    \item[(a)]$X_j \vert_{0} = \partial_j \vert_{0}$, i.e. \begin{align*}
        X_j \varphi(0) = \partial_j\varphi(0) \quad \text{ for all }\varphi\in C^{\infty}. 
        \end{align*}
    \item[(ii)] If $e_j$ denotes the $j$-th element of the canonical basis of $\R^N$, then \begin{align*}
    X_j(0)= e_j.
    \end{align*}
\end{enumerate}
\end{prop}

An $X_j$ satisfying one of the above conditions will be called a Jacobian vector field, and we take the set of elements $\mathcal{X} = \{X_1, \cdots, X_{m_1}\}$ to be the Jacobian basis of $\vf_1$. We now say the \emph{sub-Laplacian} is the differential operator $\L$ defined as, \begin{align*}
    \L = \sum_{i=1}^{m_1} X_i^2 = \sum_{i=1}^{m_1} X_i(X_i).
\end{align*}
where each $X_i$ is a part of the Jacobian basis of $\vf_1$.

We will now define the \textit{Carnot-Carath\'{e}odory distance} on $\G$. This definition is taken from \cite[Definition 5.2.2]{BLU}. Let $\mathcal{X} = \{X_1,\cdots, X_{m_1}\}$ be the Jacobian basis for $\vf_1$. Then a piecewise smooth path $\gamma : [0,T] \rightarrow \G$ is called \textit{$\mathcal{X}$-subunit} if \begin{equation}\label{subunit}
    \langle \dot{\gamma}(t) ,\xi \rangle^2 \leq \sum_{i=1}^{m_1} \langle X_jI(\gamma(t)), \xi \rangle^2 \quad \text{ for all }\xi \in \R^N
\end{equation}
and almost everywhere (in the sense of Lebesgue measure on $\R$) $t \in [0,T]$. Set $S(\mathcal{X})$ to be the set of all $\mathcal{X}$-subunit paths and put \begin{align*}
    \ell(\gamma) = T
\end{align*}
if $[0,T]$ is the domain of $\gamma \in S(\mathcal{X})$. The Carnot-Carath\'{e}odory (CC) distance is then, \begin{equation}\label{cc-distance}
d(x,y) := \inf \left\{ \ell(\gamma): \gamma \in S(\mathcal{X}), \gamma(0) = x, \gamma(T)= y\right\}
\end{equation}
and the corresponding metric is, \begin{equation}\label{cc-metric}
 \| y^{-1}\cdot x\| := d(x,y).
\end{equation}
Open balls are then taken as $B(x,r) := \{y : d(y,x) < r\}$. The CC-distance is symmetric, i.e. $d(x,y) = d(y,x)$, which also implies $\| p^{-1}\| = \| p\|$ for all $p \in \G$. For $t > 0$ define $\delta_t : \g \rightarrow \g$ by $\delta_t(X) = t^i X$ if $X \in \vf_i$ and extend to all of $\g$ by linearity. By conjugation with the exponential map, $\delta_t$ induces an automorphism on $\G$ which we also denote by $\delta_t$. Then $(\delta_t)_{t>0}$ is the one-parameter family of \textit{dilations} on $\G$ which satisfy $d(\delta_t(x),\delta_t(y)) = td(x,y)$ for all $x,y \in \G$. Note that this implies $\| \delta_t(p)\| = t\|p\|$ for all $p\in \G$. The Jacobian determinant of $\delta_t$ is equal to $t^Q$ where \begin{align}
Q = \sum_{i=1}^s i \dim \vf_i
\end{align}
is the \textit{homogeneous dimension} of $\G$. We then have $\G$ is diffeomorphic with $\g = \R^N$, $N = \sum_{i=1}^s, \dim \vf_i$ by the exponential map. The Haar measure on $\G$ is induced by the exponential map from Lebesgue measure on $\g$ and the Haar measure agree, up to a constant, with the $Q$-dimensional Hausdorff measure in the metric space $(\G, d)$. In this paper, we will work almost exclusively with the CC-distance which we do simply because it satisfies the triangle inequality.

In the case of Lemma \ref{real-ibp}, the proof uses a different metric. The sub-Laplacian has a unique fundamental solution $\cg$. The proofs of the below facts can be found in \cite[Section 5]{BLU}. 
\begin{prop}
For all $t > 0$ and $p \in \G \backslash \{0\}$, \begin{enumerate}
    \item[(i)] $\cg(p^{-1}) = \cg(p)$,
    \item[(ii)] $\cg(\delta_t(p))= t^{2-Q}\cg(p)$,
    \item[(iii)] $\cg(p) > 0$.
\end{enumerate} 
\end{prop}
These facts allow us to define a norm in terms of $\cg$. The function, 
\begin{equation}
\| p\|_{\cg}=
    \begin{cases}
        \cg(p)^{1/(2-Q)} & \text{if } p \in \G \backslash \{0\}\\
        0 & \text{if } p =0
    \end{cases} \label{cg-norm}
\end{equation}
defines a symmetric homogeneous norm which is $C^{\infty}$ away from the origin. Symmetric means $\| p^{-1}\|_{\cg} = \|p\|_{\cg}$ for all $p \in \G$. Now let, \begin{align}
    d_{\cg}(x,y) = \| y^{-1}\cdot x\|_{\cg} \label{cg-dist}
\end{align}
be the left-invariant quasi-distance induced by $\| \cdot \|_{\cg}$.

Taking $\mathcal{X} = \{X_1, \cdots , X_{m_1}\}$ to be the Jacobian basis of $\vf_1$, we define the \emph{horizontal gradient} by \begin{equation}\label{horiz-gradient}
    \nabla_{\G}f = (X_1f, \cdots , X_{m_1}f)
\end{equation}
where $f$ is some real-valued function defined on an open set of $\G$. Moreover, we can define the \emph{horizontal divergence} by \begin{align*}
\text{div}_{\G} F = X_1 F_1 + \cdots + X_{m_1}F_{m_1}   
\end{align*} 
where $F = (F_1, \cdots, F_{m_1})$ is a vector field defined on some open set of $\G$. Lastly, the convolution operation $\ast$ will be taken to mean \begin{equation}\label{convolution}
f \ast g(x) := \int f(y) g(y^{-1}\cdot x)\,dy = \int f(x \cdot p^{-1})g(p)\,dp.    
\end{equation} 
Convolutions are not commutative since $\G$ is nonabelian. For more background on Carnot groups, see \cite{BLU} and \cite{FR}.
\end{subsection}

\vskip2em

\begin{subsection}{Left-invariant homogeneous operators and fundamental solutions.} For any $A \subset \G$ we will let $C_0^{\infty}(A)$ be the set of all $C^{\infty}$ functions with compact support contained in $A$. For $\Omega$ open denote by $\D'(\Omega)$ the space of distributions on $\Omega$ whose test space is $C_0^{\infty}(\Omega)$ and equip $\D'(\Omega)$ with the locally convex topology, and set $\D' = \D'(\G)$. When talking about distributions, the symbol $\langle T, \varphi\rangle$ will denote the pairing of a distribution $T \in \mathcal{D}'(\Omega)$ and a test function $\varphi \in C_0^{\infty}(\Omega)$.

We say that a distribution $T \in \D'(\Omega)$ is \textit{homogeneous of degree $\l$} provided $\langle T, \varphi \circ \delta_t\rangle = t^{-Q-\l}\langle T, \varphi \rangle$ for all $\varphi \in C_0^{\infty}(\Omega)$. A distribution which is $C^{\infty}$ away from $0$ and homogeneous of degree $\l-Q$ will be called a \textit{kernel of type $\l$}. A differential operator $\L$ will be called \textit{homogeneous of degree $\l$}, or $\l$-homogeneous, if \begin{equation}\label{hom-diff-op}
 \L ( T \circ \delta_t) = t^{\l}(\L T)\circ \delta_t\,\,\,\,\text{ for all }T \in \D'. 
\end{equation}

Furthermore, a differential operator $\L$ is said to be \textit{hypoelliptic} if for any open $\Omega \subset \G$ and any two distriubions $T$ and $F$ on $\Omega$ satisfying $\L T = F$, then $F \in C^{\infty}(\Omega)$ implies $T \in C^{\infty}(\Omega)$. We call a function $f$ satisfying $\L f = 0$ an $\L$-solution.

A fundamental result in sub-Riemannian geometry, due to Folland, deals with the existence of fundamental solutions for such differential operators. Let $\L^t$ denote the transpose of a differential operator $\L$.
\begin{thm}\cite[Theorem 2.1]{Fo}\label{folland-fund-soln} Let $\L$ be a differential operator on a Carnot group. Suppose $\L$ is homogeneous of degree $\l$, $0 < \l < Q$, such that both $\L$ and $\L^t$ are hypoelliptic. Then there exists a unique kernel $K$ of type $\l$ which is a fundamental solution of $\L$ at $0$, i.e. $\L K = \delta$.
\end{thm}

For $\a, x \in \G$ define $\tau_{\a} : \G \rightarrow \G$ by $\tau_{\a}(x) := \a \cdot x$. Note that by definition of a Lie group, $\tau_{\a}$ is smooth. Then a differential operator $\L$ is \textit{left-invariant} if \begin{equation}\label{left-inv-op}
\L (T \circ \tau_{\a}) = (\L T)\circ \tau_{\a} \, \quad \text{ for all }\a \in \G \text{ and }T \in \D'    
\end{equation}
\begin{exmp}
Let $\{X_1, \cdots, X_{m_1}\}$ be a basis for $\vf_1$. Then $\L := \sum_{i=1}^{m_1}X_i^2$ is called a \textit{sub-Laplacian} on $\G$. It is easily checked that $\L$ is left-invariant, homogeneous of degree $2$, and self-adjoint. Furthermore, by H\"{o}rmander's theorem \cite[Theorem $1$ in the preface]{BLU}, $\L$ is hypoelliptic.   
\end{exmp}

\begin{defn}\label{removable-sets-defn}
Let $\L$ be a linear partial differential equation and $\mathcal{F}$ a space of distributions. A set $K \subset \G$ is \textit{removable for $\L$-solutions in $\mathcal{F}$} if for any domain $\Omega$ and any element $f \in \mathcal{F}(\Omega)$ satisfying $\L f = 0$ in $\Omega \backslash K$ also satisfies $\L f = 0$ in $\Omega$.
\end{defn}

\begin{rem}
Notice that this definition of removability does not assume $K$ is a subset of $\Omega$. This adds some interesting complexities in understanding removable sets. For example, a natural question one may ask is if removability is monotonic? In other words, if $E$ is removable (with respect to a class of distributions and a PDE) and $F \subset E$, does this imply $F$ is also removable? When $K$ is not assumed to be a subset of $\Omega$ as in Definition \ref{removable-sets-defn}, it is not clear from the definition if monotonicity holds. However, for the differential operators and function spaces we consider, the results in this paper will in fact show that removability is monotonic.
\end{rem}

We record a simple, but very useful, proposition and the proof of which can be found in \cite[Page $32$]{BLU}.
\begin{prop}\label{hom-degree-of-deriv}
Let $f$ be a smooth homogeneous function of degree $m \in \R$ and $X$ be a linear differential operator of degree $n \in \R$. Then $Xf$ is a homogeneous function of degree $m-n$.
\end{prop}

\begin{rem}\label{deriv-of-fund-soln}
It is simple to check that the same proof which gives Proposition \ref{hom-degree-of-deriv} also works when we relax the smoothness condition to $f \in C^{\infty}(\G \backslash \{0\})$. In particular, if $k$ is a kernel of type $\l$ and $X$ a linear differential operator of degree $m$, then since $k$ is a smooth function on $\G \backslash \{0\}$, we have $X k$ is both smooth and homogeneous of degree $\l - Q - m$ on $\G \backslash \{0\}$. This follows since $k$ is smooth and homogeneous of degree $\l - Q$ on $\G \backslash \{0\}$.
\end{rem}
The following is well known for those who study sub-Riemannian geometry, but we record the proof anyway. 
\begin{prop}\label{hom-func-bound}
Suppose $f$ is both homogeneous of degree $\l$ and continuous on $\G \backslash \{0\}$. Then, $$|f(p)| \leq C \| p\|^{\l}$$ for a constant $C>0$ which only depends on $f$ and $\G$. 
\end{prop}
\begin{proof}
Set $\Sigma := \{ p : \| p\| =1\}$. Since the topology on $\G$ coincides with the topology on $\R^N$, we get $\Sigma$ is compact as $\Sigma$ is closed and bounded. Furthermore, $\Sigma$ avoids the origin. Hence, by continuity, there exists $p_0 \in \Sigma$ such that $|f(p)| \leq |f(p_0)| =: C$ for all $p \in \Sigma$. Now fix any $p \in \G \backslash \{0\}$. Then by homogeneity of $f$, \begin{align*}
 |f(p)| &= |f(\delta_{\|p\|} \circ \delta_{\|p\|^{-1}}(p))| \\ &= \| p\|^{\l} |f(\delta_{\|p\|^{-1}}(p))| \\ &\leq  C \| p\|^{\l}
\end{align*}
which completes the proof.
\end{proof}

Lastly, some parts of this paper will use a specific kind of multi-index notation and product rule. Let $m_1$ be as in (\ref{dim-of-first-layer}). For a multi-index $\a = (\a_1,\cdots , \a_{\ell}) \in \{1,\cdots , m_1\}^{\ell}$ write \begin{equation}\label{multi-index-notation}
    X_{\a} = X_{\a_1}\cdots X_{\a_{\ell}} \quad \text{ and } \quad |\a| = \ell.
\end{equation} We describe a product rule, the main purpose of which is track the homogeneity degree of the operators involved. Fix $\a \in \{1, \cdots , m_1\}^{\ell}$. For $\beta = (\beta_1, \cdots, \beta_{\ell})$ with $\beta_i \in \{0,1\}$ for all $i$, set \begin{equation}\label{notation_for_X_B}
X_{\a}^{\beta} = X_{\a_1}^{\beta_1} \cdots X_{\a_{\ell}}^{\beta_{\ell}}.\end{equation} We make the convention of $X_{\a_i}^0$ to be the identity operator $I$, i.e. no derivatives are applied. Then $X_{\a_i}^{\beta_i} = I$ when $\beta_i = 0$ and $X_{\a_i}^{\beta_i} = X_{\a_i}$ when $\beta_i = 1$. Therefore we get the following: \begin{equation}\label{length-of-beta-index}
X_{\a}^{\beta} \text{ is of length }  \sum_{i=1}^{\ell}\beta_i.    
\end{equation}
Now for $\psi, \varphi \in C^{\infty}$, it can be checked that the product rule reads, \begin{equation}\label{product_rule}
X_{\a}(\psi \varphi) = \sum_{\beta \in \mathcal{B}}X_{\a}^{\beta}(\psi)X_{\a_1}^{1-\beta_1} \cdots X_{\a_{\ell}}^{1-\beta_{\ell}}(\varphi).    
\end{equation}
As a short hand, for a fixed $\beta$ set $\star{X}_{\a}^{\beta} = X_{\a_1}^{1-\beta_1} \cdots X_{\a_{\ell}}^{1-\beta_{\ell}}$ and $\overline{\beta} = \{ 0\}^{\ell}$. We lastly note that for $X_{\a}$ with $|\a| = \ell$, there are $2^{\ell}$ possible $\beta$'s to pick from.
\end{subsection}

\vskip2em

\section{Capacity and Removability for Lipschitz Harmonic Functions}\label{lip-section}
The only differential operator we will consider in this section is the sub-Laplacian. Let $\mathcal{X} = \{X_1, \cdots , X_{m_1}\}$ be the Jacobian basis for $\vf_1$ (see Section \ref{section-2} for details). Recall that the \emph{sub-Laplacian} is then, \begin{equation}\label{sub-laplacian}
    \L := \sum_{i=1}^{m_1} X_i^2 = \sum_{i=1}^{m_1}X_i(X_i).
\end{equation}
Then $\L$ is left-invariant, homogeneous of degree $2$, and self-adjoint. By H\"{o}rmander's theorem \cite[Theorem 1 in the preface]{BLU}, $\L$ is hypoelliptic. Hence, there exists a kernel of type $2$, denoted $\cg$, which is a fundamental solution to $\L$. For a function $f$, define the \textit{horizontal gradient} by, 
\begin{equation}\label{horiz-grad}
\nabla_{\G}f := (X_1 f, \cdots, X_{m_1}f)    
\end{equation}
whenever the right hand side makes sense.

\begin{defn}[Lipschitz Functions]
For an open set $\Omega \subset \G$ let $f : \Omega \rightarrow \R$ be a function. We say that $f \in \text{Lip}(\Omega)$ if there exists a constant $C>0$ such that, \begin{align*}
    |f(x) - f(y)| \leq C d(x,y) \text{ for all }x,y\in \Omega.
\end{align*}
Denote by $\text{Lip}(f)$ the smallest such $C$, and set $\text{Lip} = \text{Lip}(\G)$.
\end{defn}

\begin{defn}[Local Lipschitz Functions]
Let $f: \G \rightarrow \R$ be a function. We say that $f \in \text{Lip}_{\text{loc}}$ if for every compact set $E$, there exists a constant $C = C(E) >0$ such that $$|f(x)- f(y)| \leq Cd(x,y) \quad \text{ for all }x,y\in E.$$   
\end{defn}

\begin{defn}[Lipschitz Capacity]
The capacity of a bounded $K \subset \G$ with respect to $\text{Lip}$ and the sub-Laplacian $\L$ (as given in (\ref{sub-laplacian})) is defined as, 
\begin{align*}
 \k(K) := \sup \{ |\langle \L f, 1 \rangle |: f \in \text{Lip}_{\text{loc}}, \| \nabla_{\G}f\|_{L^{\infty}} \leq 1, \supp(\L f) \subset K\}. 
\end{align*}    
\end{defn}

\begin{rem}
The definition of Lipschitz capacity was originally given by Paramonov \cite[Definition 2.1]{P} in Euclidean space with the extra condition $\nabla f(\infty) = 0$, where $\nabla$ is the Euclidean gradient. This extra condition allows one to rewrite distributions in a convenient way. More specifically, one can get the distributional equality of, $$f = \Phi \ast \Delta f +c$$ where $\Delta$ is the Laplacian, $\Phi$ is the fundamental solution to $\Delta$, and $c$ is a constant. This relationship will not be needed for our purposes, so we omit the condition $\nabla f(\infty) = 0$.
\end{rem}

Corollary \ref{bounded-deriv-is-lip} says that if a function has a bounded horizontal gradient, then it is Lipschitz. The following lemma will help us in showing this and it is essentially one half of \cite[Lemma 3.1]{LDPS}.
\begin{lem}\label{local-implies-global}
 Let $f : \G \rightarrow \mathbb{R}$ be a locally Lipschitz function with $\| \nabla_{\G} f\| \leq C$ for some finite constant $C$. Then, $f$ is Lipschitz and \begin{align*}
     Lip(f) \leq \| \nabla_{\G} f\|_{L^{\infty}} \leq C.
 \end{align*}
\end{lem}
\begin{proof}
Fix $x$ and $y$. Recall that $d(x,y)$ is defined as an infimum over subunit paths, $\gamma$, joining $x$ and $y$. By \cite[Theorem 5.15.5]{BLU}, this infimum is actually a minimum. So find $\gamma :[0,T] \rightarrow \G$ for which $T= d(x,y)$ and $\gamma(0)=x$, $\gamma(T)=y$. 
Set $\text{Im}(\gamma) = \{ \gamma(t): t \in [0,T]\} \subset \G$. Since $\gamma$ is continuous, $\text{Im}(\gamma)$ is compact. As $f$ is locally Lipschitz, we have $M := \text{Lip}(f \restrict \text{Im}(\gamma)) < \infty$.
Now we show $f \circ \gamma$ is Lipschitz on $[0,T]$. Take $a,b \in [0,L]$. Then, \begin{align*}
|f \circ \gamma(a) - f\circ \gamma(b)| \leq Md(\gamma(a),\gamma(b)) \leq M \text{Lip}(\gamma)|a-b|
\end{align*}
which shows that $f \circ \gamma$ is Lipschitz. Hence, \begin{align*}
    |f(x)-f(y)| = |f\circ \gamma(0) - f\circ \gamma(T)| &= \left|\int_0^T (f \circ \gamma)'(s) \,ds  \right|\\ &= \left|\int_0^T \langle f(\gamma(s)), \ \dot{\gamma}(s)\rangle\, ds  \right| \\&\overset{(\ref{subunit})}{\leq} \int_0^T \left( \sum_{j=1}^{m_1} (\langle X_jI(\gamma(s)), (\nabla f)(\gamma(s))\rangle)^2  \right)^{1/2}\,ds \\ &= \int_0^T  \left(  \sum_{j=1}^{m_1} [(X_j f)(\gamma(s))]^{2} \right)^{1/2} \\ &\leq \| \nabla_{\G}f\|_{L^{\infty}} T \\ &= \| \nabla_{\G}f\|_{L^{\infty}} d(x,y)
    \end{align*}
which finishes the proof.
 %\\ &\leq T \sup\{ |(f\circ \gamma)'(t)|: t\in [0,L]\} \\ &\leq T  \text{Lip}(\gamma)\| \nabla_{\G}f\|_{L^{\infty}} \\ &\leq TC = C d(x,y)
\end{proof}

\vskip2em

\begin{cor}\label{bounded-deriv-is-lip}
Let $f : \G \rightarrow \mathbb{R}$ be a continuous function satisfying $\| \nabla_{\G} f\|_{L^{\infty}} \leq C$ for some finite constant $C$. Then $f$ is Lipschitz.  
\end{cor}
\begin{proof}
By [Theorem 1.4, \cite{GN}], $f$ is locally Lipschitz. Then using Lemma \ref{local-implies-global}, we get $f$ is Lipschitz. 
\end{proof}

Given Theorems \ref{intro-cap} and \ref{intro-hold}, one might be tempted to prove $\h^{Q-1}(K) \sim \k(K)$ for a compact $K$ then show removability as a consequence. This explicitly fails though. Indeed, by \cite[Theorem 1.2]{CMT}, there exists a compact set $K \subset \G$ for which $\mathcal{H}_{\infty}^{Q-1}(K) > 0$ but $\k(K) = 0$. Instead, we can prove removability through a more direct approach. We will show how the proof of the removability characterization goes, then prove the necessary lemmas afterwards.

\begin{thm}[Removability Characterization for Lip]\label{rem-for-lip}
Let $K \subset \G$ be bounded. Then $K$ is removable for $\L$-solutions in Lip if and only if $\k(K) =0$.    
\end{thm}
\begin{proof}
Assume $K$ is removable and take $f \in \text{Lip}_{\text{loc}}$ which satisfies $\supp (\L f)\subset K$ distributionally. Applying the definition of removability to the domain $\Omega = \G$, we automatically get $\L f \equiv 0$ on $\G$. This gives $\k(K) = 0$.

For the other direction, we will use the contrapositive. Assume $K$ is not removable for $\L$-solutions in $\text{Lip}$. Then there exists a domain $D$ and a locally Lipschitz function $f: D \rightarrow \R$ satisfying $\L f = 0$ in $D \backslash K$ but $\L f \neq 0$ in $D$. So $f$ is not an $\L$-solution in the distributional sense which means there exists $\varphi \in C_0^{\infty}(D)$ with $|\langle \L f, \varphi \rangle| = |\langle \varphi \L f, 1 \rangle| > 0$. Notice $\supp( \varphi \L f) \subset K$. Set $f_{\varphi} = (\varphi \L f) \ast \cg$. Left-invariance of $\L$ gives, \begin{align*}
    \L(f_{\varphi}) = \varphi \L f \ast \L \cg = \varphi \L f
\end{align*} so that $\supp ( \L f_{\varphi}) \subset K$. By Lemmas \ref{first_func_is_lip}, \ref{second_func_in_lip}, and \ref{third_func_in_lip} we get $f_{\varphi}\in \text{Lip}$ and $\| \nabla_{\G} f_{\varphi}\|_{L^{\infty}}=: A < \infty$. Hence, 
$$\k(K) \geq |\langle \L (f_{\varphi}/A), 1 \rangle| = |\langle \varphi \L f, 1 \rangle|/A = |\langle \L f , \varphi \rangle| /A > 0$$  which completes the proof.
%Then there exists a compact $F \supset K$, a domain $D$, and a locally Lipschitz function $f : D \rightarrow \mathbb{R}$ satisfying $\L f = 0 $ in $D \backslash F$ but $\L f \neq 0$ in $D$.
\end{proof}

Now we show the necessary lemmas which give $f_{\varphi} = \varphi \L \ast \cg$ is Lipschitz with a bounded horizontal gradient. We first start by getting a distributional formula for objects of the form $\varphi \L f \ast \cg$. Let $f$ be a locally Lipschitz function and take $\psi, \varphi \in C_0^{\infty}$. Since $\L$ is self-adjoint, $\langle \varphi \L f, \psi \rangle = \langle f, \L(\psi  \varphi)\rangle$. Moreover, by \cite[Exercise 6 of Chapter 1]{BLU}, \begin{align*}
\L (\psi \varphi) = \psi \L \varphi + 2\langle \nabla_{\G} \psi, \nabla_{\G} \varphi\rangle + \varphi \L \psi.    
\end{align*}
Since the horizontal gradient of a real-valued function generates a vector in $\R^{m_1}$, the inner product above is the dot product on $\R^{m_1}$.
These two equalities now give,
\begin{align*}
    \langle \varphi \L f, \psi \rangle &= \langle f , \psi \L \varphi \rangle + \langle f, 2\sum_{i=1}^{m_1} X_i \psi X_i \varphi\rangle + \langle f, \psi \L \varphi\rangle \\ &= \langle \L(f \varphi), \psi \rangle - \langle \sum_{i=1}^{m_1} 2 X_i(f X_i \varphi), \psi\rangle+ \langle f \L \varphi, \psi\rangle
\end{align*}
So distributionally, \begin{align*}
 \psi \L f  = \L(f \psi) - 2\sum_{i=1}^{m_1} X_i(fX_i \psi) + f\L \psi.   
\end{align*}
Therefore, by \cite[Corollary 2.8]{Fo},
\begin{equation}\label{dist-form-lip}
\varphi \L f \ast \cg = f \varphi -2 \sum_{i=1}^{m_1} X_i(f X_i \varphi )\ast \cg + f \L \varphi \ast \cg. 
\end{equation}
Hence, it suffices to show that each function on the right hand side of (\ref{dist-form-lip}) is Lipschitz and has a bounded horizontal gradient. The easiest cases are showing this for $f \varphi$ and $f \L \varphi \ast \cg$, so we start with these.
\begin{lem}\label{first_func_is_lip}
Take $f \in \text{Lip}_{\text{loc}}$ and $\varphi \in C_0^{\infty}(B(a,\delta))$. Then $f\varphi \in \text{Lip}$ and $\| \nabla_{\G} (f\varphi)\|_{L^{\infty}} < \infty$.    
\end{lem}
\begin{proof}
The stratified mean value theorem, i.e. \cite[Theorem 20.3.1]{BLU}, says that for $\psi \in C^1$ there exists constants $c$ and $b$ which only depend on $\G$ and the CC-metric such that \begin{align*}
    |\psi (p \cdot h) - \psi(p)| \leq c d(h) \sup_{z : d(z)\leq b d(h)} |(X_1\psi(p \cdot z), \cdots, X_{m_1}\psi(p \cdot z))|.
\end{align*}
Therefore, $\varphi$ is Lipschitz. Now we show $f \varphi$ is Lipschitz. Take $x,y \in \G$. If both $x \notin B(a,\delta)$ and $y\notin B(a,\delta)$ then clearly $|f(x) \varphi(x) - f(y)\varphi(y)|= 0 $ and we're done. Now assume $x,y \in B(a,\delta)$. Set $M = \text{Lip}(f \restrict B(a,\delta))$. Then, \begin{align*}
    |f(x)\varphi(x) - f(y) \varphi(y)| &\leq |f(x)\varphi(x) - f(y) \varphi(x)| + |f(y)\varphi(x) - f(y) \varphi(y)| \\ &\leq \| \varphi\|_{\infty} Md(x,y) + \| f\|_{L^{\infty}(B(a,\delta))} \text{Lip}(\varphi) d(x,y)
\end{align*}
which finishes the case when $x,y \in B(a,\delta)$. For the final case, assume without loss of generality that $x\in B(a,\delta)$ and $y \notin B(a,\delta)$. Since $\varphi \in \text{Lip}$ and $\varphi(y) = 0$, \begin{align*}
    |f(x) \varphi(x) - f(y)\varphi(y)|  &\leq \| f\|_{L^{\infty}(B(a,\delta))} |\varphi(x)| \\ &= \|f \|_{L^{\infty}(B(a,\delta))}|\varphi(x) - \varphi(y)| \\ &\leq \|f \|_{L^{\infty}(B(a,\delta))}\text{Lip}(\varphi) d(x,y)
\end{align*}
so therefore, in all cases we have $f\varphi \in \text{Lip}$.

Set $g = f \varphi \in \text{Lip}$. By Pansu's differentiability theorem \cite{Pan}, $X g(x)$ exists for almost every $x \in \G$, where $X$ is any element in $\mathcal{X}$. Furthermore, \cite[Lemma 3.1]{LDPS} says that for a Lipschitz $f: \G \rightarrow \R$, \begin{align*}
    \text{Lip}(f) \geq \sup\{ |X f(x)|: X \in \mathcal{X}, \, x \in \G, \, Xf(x) \text{ exists}\}.
\end{align*} 
Hence, $\| \nabla_{\G} g\|_{L^{\infty}} < \infty$.
\end{proof}

We now move onto $f \L \varphi \ast \cg$ by girst giving a decay estimate.

\begin{prop}\label{decay-of-one-deriv}
 Set $\xi(x) = f \L \varphi \ast X\cg(x)$ for $f \in \text{Lip}_{\text{loc}}$ and $\varphi \in C_0^{\infty}(B(a,\delta))$, and where $X$ is any element in $\mathcal{X}$. Then there exists $R < \infty$ such that \begin{align*}
     |\xi(x)|\leq 1 \quad \text{ for all }x\in B(a,R)^c.
 \end{align*}  
\end{prop}
\begin{proof}
It is easy to check that since $\cg$ is homogeneous of degree $2-Q$ on $\G \backslash \{0\}$ that $\cg$ is also homogeneous of degree $2-Q$ on $\G \backslash \{0\}$. So Remark \ref{deriv-of-fund-soln} and Proposition \ref{hom-func-bound} say $|X \cg(p)|\lesssim \| p\|^{1-Q}$ for all $p \in \G \backslash \{0\}$. Set $h = f \L \varphi$ so that $\supp (h) \subset B(a,\delta)$ and $M := \| h\|_{\infty} < \infty$. Then,  \begin{align}
    |\xi(x)| &= \left| \int_{B(a,\delta)}h(y) X\cg(y^{-1}\cdot x)\,dy\right| \nonumber \\ &\leq M \int_{B(a,\delta)}\| y^{-1}\cdot x\|^{-Q}\,dy. \label{step1}
\end{align}
The CC-distance satisfies the triangle inequality and is symmetric, so $\| p^{-1}\cdot q\| \geq \| q\| -\|p\|$ for all $p,q \in \G$. Take any $x \in B(a,2\delta)^c$. Then for $y \in B(a,\delta)$, clearly $$\| y^{-1}\cdot x\| = \| y^{-1}\cdot a \cdot a^{-1}\cdot x\| \geq \|x^{-1}\cdot a\|/2.$$ Resuming from (\ref{step1}) then gives, \begin{align*}
|\xi(x)| \lesssim M \|x^{-1}\cdot a\|^{-Q} \delta^Q
\end{align*}
which completes the proof since $\delta$ and $M$ are fixed.
\end{proof}

\begin{lem}\label{third_func_in_lip}
Take $f \in \text{Lip}_{\text{loc}}$ and $\varphi \in C_0^{\infty}(B(a,\delta))$. Then the function $\Theta$ defined by $\Theta (x) := f\L \varphi \ast \cg(x)$ is continuous and satisfies $\| \nabla_{\G} \Theta \|_{L^{\infty}} < \infty$. In particular, Corollary \ref{bounded-deriv-is-lip} gives $\Theta \in \text{Lip}$.  
\end{lem}
\begin{proof}
%Like the proof of Lemma \ref{second_func_in_lip}, it suffices to show $\Theta \in C^1 \left( B(a,2\delta)\right)$. Take any $X \in \mathcal{X}$. Remark \ref{deriv-of-fund-soln} and Proposition \ref{hom-func-bound} say $|X \cg (p)|\lesssim \| p\|^{1-Q}$ for all $p \in \G \backslash \{0\}$. Let $g = f \L \varphi$. Then for any $x \in B(a,2\delta)$, left-invariance of $X$ yields, \begin{align}
%|X\Theta(x)| &\lesssim \|g\|_{\infty} \int_{B(a,\delta)} \| y^{-1}\cdot x\|^{1-Q}\,dy \nonumber \\ &\sim \sum_{j=0}^{\infty} \int_{ 2^{-j-1}(3\delta) < \| x^{-1}\cdot y\| \leq 2^{-j} (3\delta)} \|x^{-1}\cdot y\|^{1-Q}\,dy \nonumber \\ &\lesssim \sum_{j=0}^{\infty} (2^{-j-1}(3\delta))^{1-Q} \left| B(x, 2^{-j}(3\delta)) \right| \nonumber \\ &\sim \delta \sum_{j=0}^{\infty} 2^{-j} = \delta. \label{dct-bound-2}
%\end{align}
%A change of variables gives, \begin{equation}\label{dct-cont-2}
% X\Theta(x) = \int g(x \cdot p) X\cg(p)\,dp   
%\end{equation}
%and since $g$ is continuous, (\ref{dct-bound-2}) and (\ref{dct-cont-2}) allow the use of the dominated convergence theorem to conclude $\Theta \in C^1\left( \overline{B(a,\delta)} \right)$. 
Proposition \ref{decay-of-one-deriv} gives the existence of $R > 0$ such that \begin{align}
    |X \Theta(x)| \leq 1 \quad \text{ for all }x \in B(a,R)^c. \label{bound_outside_ball}
\end{align} 
Remark \ref{deriv-of-fund-soln} and Proposition \ref{hom-func-bound} say $|X \cg(p)|\lesssim \| p\|^{1-Q}$ for all $p \in \G \backslash \{0\}$
So for any $x \in B(a,R)$, integrating over annuli gives \begin{align*}
    |X \Theta(x)|  &\lesssim \| g\|_{L^{\infty}(B(a,\delta))} \int_{B(a,\delta)} \| y^{-1}\cdot x\|^{1-Q}\,dy  \\ &\leq \|g\|_{L^{\infty}(B(a, \delta))} \sum_{j=0}^{\infty} \int_{ 2^{-j-1}(R+\delta) \leq \| y^{-1} \cdot x\| < 2^{-j}(R+\delta)}\| y^{-1} \cdot x\|^{1-Q}\,dy \\ &\lesssim \| g\|_{L^{\infty}(B(a,\delta))}\sum_{j=0}^{\infty} (2^{-j} (R+\delta))^{1-Q} \left| B(x, 2^{-j}(R+\delta))\right| \\ &\sim \| g\|_{L^{\infty}(B(a,\delta))} (R+\delta)^{1-Q} \sum_{j=0}^{\infty} 2^{j(Q-1)} (R+\delta)^Q 2^{-jQ} \\ &= \| g\|_{L^{\infty}(B(a,\delta))}(R+\delta) \sum_{j=0}^{\infty} 2^{-j} \\ &\sim \| g\|_{L^{\infty}(B(a,\delta))}(R+\delta).
\end{align*}
Along with (\ref{bound_outside_ball}), this shows $\| \nabla_{\G} \Theta \|_{L^{\infty}} \lesssim 1$.

For continuity, simply note $\supp \L \varphi$ is compact, $\cg$ is in $L_{\text{loc}}^1$, and $f \L \varphi$ is a continuous function. So the dominated convergence theorem gives that $\Theta$ is continuous.

\end{proof}

We can now move to the hardest case. For some intuition, we will be applying two derivatives to $\cg$ which will force the resulting function to be homogeneous of degree $-Q$. However, these functions need not be locally integrable (around the origin) and this is what causes issues. The saving grace though comes with our Lipschitz function which will allow us to raise the exponent to $1-Q$, and this is locally integrable.

First, we show an integration by parts formula. 

\begin{lem}\label{real-ibp}
Suppose $g$ is a continuous function satisfying $|\nabla_{\G} g\|_{L^{\infty}} \lesssim 1$ and $\psi \in C_0^{\infty}$. Fix $x \in \G$. Then for a differential operator $Y$ homogeneous of degree $1$ and a horizontal vector field $X$, we have \begin{align*}
 \int (Xg)(y)\psi(y)(Y \cg)(y^{-1}\cdot x)\,dy &= \int (g(y) - g(x)) \psi(y) (X \wt{Y}\wt{\cg})(x^{-1}\cdot y)\,dy \\ &\,\,\,\,\,\,\,\,\,\,\,\,\,\,\,\,\,\,\,\,\,\,\,\,\,\,\,\,\,\,\,\,\,\,- \int (g(y)-g(x)) (X \psi)(y)(Y \cg)(y^{-1}\cdot x)\,dp. 
\end{align*}
\end{lem}
\begin{proof}
For this proof only, all balls will be formed by the quasi-distance $d_{\cg}(\cdot, \cdot)$ given by (\ref{cg-dist}). We do this because balls for this quasi-distance are $C^1$ domains.

First consider when $g \in C^1$.
Let $\varepsilon > 0$. Note that $\cg(y^{-1}\cdot x)$ is a smooth function in $y$ on the set $B(x,\varepsilon)^c$. Then, 
\begin{align*}
\int (Xg)(y)\psi(y)(Y \cg)(y^{-1}\cdot x)\,dy &= \int_{\overline{B(x, \varepsilon)}^c} (Xg)(y)\psi(y)(Y \cg)(y^{-1}\cdot x)\,dy \\ &\,\,\,\,\,\,\,\,\,\,\,\,\,\,\,\,\,\,\,\,\,\,\,\,\,\,\,\,\,\,\,\,\,\,\,\,\,\,\, +\int_{B(x, \varepsilon)} (Xg)(y)\psi(y)(Y \cg)(y^{-1}\cdot x)\,dy \\ &= \int_{\overline{B(x, \varepsilon)}^c} (X)_y(g(y)-g(x))\psi(y)(Y \cg)(y^{-1}\cdot x)\,dy \\ &\,\,\,\,\,\,\,\,\,\,\,\,\,\,\,\,\,\,\,\,\,\,\,\,\,\,\,\,\,\,\,\,\,\,\,\,\,\,\, +\int_{\overline{B(x, \varepsilon)}} (Xg)(y)\psi(y)(Y \cg)(y^{-1}\cdot x)\,dy  \\ &= I + II 
\end{align*}  
where $(X)_y$ denotes the $X$ derivative in the $y$ variable. Remark \ref{deriv-of-fund-soln} and Proposition \ref{hom-func-bound} say $|Y \cg(p)|\lesssim \| p\|^{1-Q}$ for all $p \in \G \backslash \{0\}$.Integration of annuli gives, \begin{align*}
    |II| &\lesssim \| Xg\|_{L^{\infty}(B)} \| \psi\|_{L^{\infty}(B)} \int_{B(x, \varepsilon)} \| y^{-1}\cdot x\|^{1-Q}\, dy \sim \| X g\|_{L^{\infty}(B)} \| \psi\|_{L^{\infty}(B)} \varepsilon.
\end{align*}
Therefore, \begin{align}
    |II| \rightarrow 0 \,\text{ as }\,\varepsilon\rightarrow 0.\label{II-zero}
\end{align}
Without loss of generality, $X = X_1$. In other words, $X$ is the first element in the Jacobian basis of $\vf_1$. Set $\wt{g} = (g,0, \cdots, 0) \in \R^{m_1}$. Then notice, \begin{align*}
    (X)_y(g(y) - g(x)) = (\text{div}_{\G})_y(\wt{g}(y) - \wt{g}(x))
\end{align*}
where $(\text{div}_{\G})_y$ denotes the horizontal divergence in the $y$ variable. Then notice, \begin{align*}
I &= \int_{\overline{B(x, \varepsilon)}^c} (\text{div}_{\G})_y\left[ (\wt{g}(y) - \wt{g}(x))\psi(y) (Y \cg)(y^{-1}\cdot x)  \right] \,dy  \\ &\hspace{25mm}-\int_{\overline{B(x, \varepsilon)}^c} \langle \wt{g}(y) -\wt{g}(x), (\nabla_{\G})_y(\psi(y)(Y \cg)(y^{-1}\cdot x))\rangle\,dy \\ &=  I_a - I_b
\end{align*}
where $\langle \cdot, \cdot \rangle$ in the integrand of $I_b$ is the inner product on $\R^{m_1}$.
Take a ball $B$ for which $\supp \psi \subset B$ and $B(x,2 \varepsilon) \subset B$. Then $B \cap \overline{B(x, \varepsilon)}^c$ is a bounded $C^1$ domain. Moreover, $\wt{g}$ is a $C^1$ vector field. Then \cite[(3.12)]{CMT} says \begin{align*}
    I_a &= \int_{\partial B(x, \varepsilon)} \langle (\wt{g}(y) - \wt{g}(x))\psi(y) (Y \cg)(y^{-1}\cdot x), \nu \rangle \a\, d\mathcal{S}_{\cg}^{Q-1} 
\end{align*}
using $\supp \psi \subset B$, and where $\mathcal{S}_{\cg}^{Q-1}$ is the $(Q-1)$-dimensional spherical Hausdorff measure induced by (\ref{cg-dist}), $\a$ is Borel function for which $c_1 \leq \a \leq c_2$ $\mathcal{S}_{\cg}^{Q-1}$-a.e. on $\partial B(x,\varepsilon)$ with $0 \leq c_1, c_2 < \infty$, and $\nu$ is a function with $|\nu| = 1$. Let $M = \text{Lip}(g\restrict B)$.
Notice, \begin{align}
    |g(y) - g(x)| \leq M \| y^{-1}\cdot x\| \quad \text{and} \quad |Y \cg(y^{-1}\cdot x)| \lesssim \| y^{-1}\cdot x\|^{1-Q} \,\,\,\text{ for all }x,y \in B. \label{g-and-Ycg}
\end{align}
Moreover, by combining \cite[Theorem 4.1]{M} and \cite[Theorem 1.2 (1.12) $\wedge$ (1.13) $\wedge$ (1.14)]{CG} yields, \begin{align}
 \mathcal{S}_{\cg}^{Q-1}(\partial B(x, \varepsilon)) \sim \varepsilon^{Q-1}.   \label{comp-ball}
\end{align}
Therefore, \begin{align*}
    |I_a| &\overset{(\ref{g-and-Ycg})}{\lesssim} M \| \psi\|_{L^{\infty}} \int_{\| y^{-1}\cdot x\| = \varepsilon} \| y^{-1}\cdot  x\|\| y^{-1}\cdot x\|^{1-Q}\, d \mathcal{S}_{\cg}^{Q-1} = \varepsilon^{2-Q} \mathcal{S}_{\cg}^{Q-1}(\partial B(x, \varepsilon)) \overset{(\ref{comp-ball})}{\sim} \varepsilon.
\end{align*}
Hence, \begin{align}
    I_{a} \rightarrow 0\, \text{ as }\, \varepsilon \rightarrow 0. \label{ia-zero}
\end{align}

We now treat $I_b$. By definition of $\wt{g}$ and expanding out the derivative, \begin{align*}
    I_b  &= \int_{\overline{B(x,\varepsilon)}^{c}} (g(y)-g(x)) (X)_y (\psi(y) (Y \cg)(y^{-1}\cdot x))\, dy \\ &= \int_{\overline{B(x,\varepsilon)}^{c}} (g(y)-g(x))\psi(y) (X)_y(Y \cg(y^{-1}\cdot x))\,dy+\int_{\overline{B(x,\varepsilon)}^{c}} (g(y)-g(x)) (X \psi)(y)(Y \cg)(y^{-1}\cdot x)\,dy.
\end{align*}
We now need to deal with $(X)_y(Y \cg(y^{-1}\cdot x))$.
By \cite[Equation $(1.11)$ on page $26$]{FR} we have \begin{equation}\label{switch-variables}
(Y\cg)(y^{-1}\cdot x) = -(\wt{Y}\wt{\cg})(x^{-1}\cdot y)
\end{equation}
where $\wt{Y}$ is using the notation of (\ref{defn-of-right-inv-vf}). Now (\ref{switch-variables}) and left-invariance of $X$ give, 
\begin{equation}\label{two-derivs-cg}
(X)_y \left((Y \cg)(y^{-1}\cdot x)\right) = - (X \wt{Y}\wt{\cg})(x^{-1}\cdot y).
\end{equation}
We note that (\ref{two-derivs-cg}) is valid by (\ref{defn-of-left-inv-vf}). In particular, notice that the definition of $X f(p)$, for a smooth enough $f$, is independent of $p$. So once we apply the $(X)_y$ derivative to $\wt{Y} \wt{\Gamma}$, we can drop the $y$-subscript. 

It is simple to check that since $\cg$ is homogeneous of degree $2-Q$ on $\G \backslash \{0\}$ that we also have $\wt{\cg}$ is homogeneous of degree $2-Q$ on $\G \backslash \{0\}$. Furthermore, as $Y$ is $1$-homogeneous we also have $\wt{Y}$ is $1$-homogeneous. Hence, applying Remark \ref{deriv-of-fund-soln} and Proposition \ref{hom-func-bound} twice gives, \begin{align*}
    |(X \wt{Y}\wt{\cg})(p)| \lesssim \| p\|^{-Q} \quad \text{ for all }p \in \G \backslash \{0\}.
\end{align*}

We may now write, \begin{align*}
    I_b &= -\int_{\overline{B(x,\varepsilon)}^c}  (g(y) - g(x)) \psi(y) (X \wt{Y}\wt{\cg})(x^{-1}\cdot y)\, dy \int_{\overline{B(x,\varepsilon)}^c} (g(y)-g(x))(X\psi)(y) (Y\cg)(y^{-1}\cdot x)\,dy \\ &=-I_{b_1} + I_{b_2}.
\end{align*}

Let $M$ be as before. Then, \begin{align*}
    |I_{b_1}| &\lesssim M\| \psi\|_{L^{\infty}} \int_B \| y^{-1}\cdot x\| \| y^{-1}\cdot x\|^{-Q}\,dy = M \| \psi\|_{L^{\infty}} \int_B \| y^{-1}\cdot x\|^{1-Q}\,dy < \infty
\end{align*}
where the last inequality follows since the function $p \mapsto \|p\|^{1-Q}$ is in $L_{\text{loc}}^1$. Similarly, \begin{align*}
    |I_{b_2}| < \infty.
\end{align*}
Moreover, the integrands in both $I_{b_1}$ and $I_{b_2}$ are continuous on $\G \backslash \{x\}$ (which is of full measure). So by the dominated convergence theorem, \begin{align}
    I_b &\overset{\varepsilon\rightarrow 0}\longrightarrow -\int (g(y) -g(x))\psi(y)(X \wt{Y}\wt{\cg})(x^{-1}\cdot y)\,dy + \int (g(y) - g(x))(X \psi)(y) (Y \cg)(y^{-1}\cdot x)\,dy. \label{i_b-zero}
\end{align}
Combining (\ref{i_b-zero}), (\ref{ia-zero}), and (\ref{II-zero}) then yields the conclusion for when $g \in C^1$.

Now suppose $g$ is continuous and satisfies $\| \nabla_{\G} g\|_{L^{\infty}}\lesssim 1$. Take $\zeta \in C_0^{\infty}(B(0,1))$ for which $\zeta(p^{-1}) = \zeta(p)$ for all $p \in \G$, $0 \leq \zeta \leq 1$, and $\int_{\G} \zeta = 1$. Set $\zeta_{t}(p) = t^{-Q}\zeta(\delta_{1/t}(p))$ and $g_t(p) = \zeta_t \ast g(p)$. Let $r_y : \G \rightarrow \G$ and $i : \G \rightarrow \G$ be the functions defined by $r_y(x) := x \cdot y$ and $i(x) := x^{-1}$. By definition of a Lie group, both $r_y$ and $i$ are $C^{\infty}$ maps (for all $y \in \G$). Now we can write, \begin{align*}
    g_t(p) = \int_K \zeta_t \circ i \circ r_{p^{-1}}(q) g(q)\,dq
\end{align*}
where $K$ is the support of $\zeta_t \circ i \circ r_{p^{-1}}$, which is also a compact set. Moreover, for any $p \in \G$, \begin{align*}
    |g_t(p)| \leq \|\zeta_t \circ i \circ r_{p^{-1}} \|_{L^{\infty}(K)} \| g\|_{L^{\infty}(K)} |K| < \infty
\end{align*}
where $|\cdot |$ is Lebesgue measure. Hence, as $\zeta_t (p\cdot q^{-1})$ is $C^{\infty}$ in $p$, we can move the derivative under the integral to obtain that $g_t \in C^{\infty}$. So the previous case gives, \begin{align*}
\int (Xg_t)(y)\psi(y)(Y \cg)(y^{-1}\cdot x)\,dy &= \int (g_t(y) - g_t(x)) \psi(y) (X \wt{Y}\wt{\cg})(x^{-1}\cdot y)\,dy \\ &\,\,\,\,\,\,\,\,\,\,\,\,\,\,\,\,\,\,\,\,\,\,\,\,\,\,\,\,\,\,\,\,\,\,- \int (g_t(y)-g_t(x)) (X \psi)(y)(Y \cg)(y^{-1}\cdot x)\,dp.     
\end{align*}
By \cite[Proposition 1.28 (v)]{Vi}, $g_t \rightarrow g$ uniformly on compact subsets of $\G$. Hence, since each integral above is supported in $\supp \psi$, sending $t \rightarrow 0$ then gives the conclusion.
\end{proof}

For a function $h$ define $\wt{h}(x):= h(x^{-1})$ wherever defined.

\begin{prop}\label{decay-of-two-derivs}
Let $g\in \text{Lip}_{\text{loc}}$ and $\varphi \in C_0^{\infty}(B(a,\delta))$. Set $\xi(x) = g \varphi \ast Y\wt{\cg}(x)$, where $Y$ is a $2$-homogeneous differential operator. Then there exists $R < \infty$ such that $$|\xi(x)| \leq 1 \quad \text{ for all }x\in B(a,R)^c.$$
\end{prop}
\begin{proof}
The proof follows in a similar way to the proof of Proposition \ref{decay-of-one-deriv}.
%It is easy to check that since $\cg$ is homogeneous of degree $2-Q$ on $\G \backslash \{0\}$ that $\wt{\cg}$ is also homogeneous of degree $2-Q$ on $\G \backslash \{0\}$. So Remark \ref{deriv-of-fund-soln} and Proposition \ref{hom-func-bound} say $|Y \wt{\cg}(p)|\lesssim \| p\|^{-Q}$ for all $p \in \G \backslash \{0\}$. Set $h = g\varphi$ so that $\supp (h) \subset B(a,\delta)$ and $M := \| h\|_{\infty} < \infty$. Then,  \begin{align}
    %|\xi(x)| &= \left| \int_{B(a,\delta)}h(y) Y \wt{\cg}(y^{-1}\cdot x)\,dy\right| \nonumber \\ &\leq M \int_{B(a,\delta)}\| y^{-1}\cdot x\|^{-Q}\,dy. \label{step1}
%\end{align}
%The CC-distance satisfies the triangle inequality and is symmetric, so $\| p^{-1}\cdot q\| \geq \| q\| -\|p\|$. Take any $x \in B(a,2\delta)^c$. Then for $y \in B(a,\delta)$, clearly $$\| y^{-1}\cdot x\| = \| y^{-1}\cdot a \cdot a^{-1}\cdot x\| \geq \|x^{-1}\cdot a\|/2.$$ Resuming from (\ref{step1}) then gives, \begin{align*}
%|\xi(x)| \lesssim M \|x^{-1}\cdot a\|^{-Q} \delta^Q
%\end{align*}
%which completes the proof since $\delta$ and $M$ are fixed.
\end{proof}

\begin{lem}\label{second_func_in_lip}
Take $f\in \text{Lip}_{\text{loc}}$ and $\varphi \in C_0^{\infty}(B(a,\delta))$. Then the function $\Theta$ defined by $\Theta(x) := X_i(f X_i \varphi) \ast \cg(x)$ is continuous and satisfies $\| \nabla_{\G} \Theta \|_{\infty} < \infty$. In particular, Corollary \ref{bounded-deriv-is-lip} gives $\Theta \in \text{Lip}$.  
\end{lem}
\begin{proof}
%Take $x \in \left( \overline{B(a,\delta)} \right)^c$. Then left-invariance of $\L$ gives, \begin{align*}
%\L \left[(X_i(f \psi) \ast \cg\right](x) = \int X_i(f(y) \psi(y)) (\L \cg)(y^{-1}\cdot x)\,dy = 0
%\end{align*}
%as $\cg$ is the fundamental solution. By hypoellipticity of $\L$, this implies $\Theta \in C^{\infty}\left( \left(\overline{B(a,\delta)}\right)^c\right)$. Hence, for the claim $\Theta \in C^1$, it suffices to show $\Theta \in C^1\left(B(a,2\delta)\right)$. 
Set $\psi = X_i \varphi$ and take any $X \in \mathcal{X}$.
Proposition \ref{decay-of-two-derivs} says that there exists $R > 0$ such that \begin{align}
    |X \Theta(x)| \leq 1 \quad \text{ for all }x \in B(a, R)^c. \label{decay-ball}
\end{align}

Now take $x \in B(a,R)$. Left-invariance of $X$ gives, \begin{align*}
X\Theta (x) &= X_i(f \psi) \ast X \cg(x) \\ &= \int (X_if(y))\psi(y) X \cg(y^{-1}\cdot x)\,dy + \int f(y)(X_i \psi(y)) X\cg(y^{-1}\cdot x)\,dy \\ &=: I(x)+II(x).
\end{align*}
Essentially the same proof of Lemma \ref{third_func_in_lip} shows $II \in \text{Lip}$, so we will only show $I(x) \in \text{Lip}$. 

The same reasoning as (\ref{switch-variables}) and (\ref{two-derivs-cg}) gives, \begin{align}
    (X_i)_y (X \cg)(y^{-1}\cdot x) = - (X_i \wt{X}\wt{\cg})(x^{-1}\cdot y). \label{dch}
\end{align}
 Inputting (\ref{dch}) into $I(x)$ and using integration by parts, i.e. Lemma \ref{real-ibp}, we get \begin{align*}
I(x) &= \int (X_i)_y(f(y)-f(x)) \psi(y) X \cg(y^{-1}\cdot x)\,dy
\\ &=\int (f(y) -f(x)) X_i \psi(y) X\cg(y^{-1}\cdot x) \,dy - \int (f(y) -f(x))\psi(y)(X_i \wt{X}\wt{\cg})(x^{-1} \cdot y)\,dy \\ &=: I_a(x) - I_b(x).
\end{align*}
 Again, Lemma \ref{third_func_in_lip} shows $I_a \in \text{Lip}$, so we only focus on $I_b$.

Remark \ref{deriv-of-fund-soln} and Proposition \ref{hom-func-bound} says $|(X_i \wt{X}\wt{\cg})(p)| \lesssim \| p\|^{-Q}$ for all $p \in \G \backslash \{0\}$. Since $f \in \text{Lip}_{\text{loc}}$ and $\supp(\psi) \subset B(a,\delta)$ we therefore have, \begin{align}
|I_b(x)| &\lesssim \text{Lip}(f \restrict \supp \psi)\| \psi\|_{\infty}\int_{B(a,\delta)} \| x^{-1}\cdot y\|^{1-Q}\,dy \nonumber \\ &\lesssim \sum_{j=0}^{\infty} \int_{ 2^{-j-1}(R+\delta) < \| x^{-1}\cdot y\| \leq 2^{-j} (R+\delta)} \|x^{-1}\cdot y\|^{1-Q}\,dy \nonumber \\ &\lesssim \sum_{j=0}^{\infty} (2^{-j-1}(R+\delta))^{1-Q} \left| B(x, 2^{-j}(R+\delta)) \right| \nonumber \\ &\sim (R +\delta) \sum_{j=0}^{\infty} 2^{-j} = R+\delta \label{dct-bound-1}
\end{align}
%Furthermore, a change of variables gives, \begin{align}
%I_b(x) &= \int (f(x \cdot p)-f(x))\psi(x\cdot p) (X_i \wt{X}\wt{\cg})(p)\,dp \label{dct-cont-1}
%\end{align}
%and everything with an $x$ in the integrand is continuous. Hence, (\ref{dct-bound-1}) and (\ref{dct-cont-1}) allow the use of the dominated convergence theorem to say $I_b(x)$ is continuous on $B(a,2\delta)$. 

Combined with (\ref{decay-ball}), this shows $\| \nabla_{\G} \Theta\|_{\infty} < \infty$. Lastly, the same reasoning as at the end of the proof of Lemma \ref{third_func_in_lip} gives that $\Theta$ is continuous. The proof is complete.
\end{proof}

\vskip2em

\section{Dyadic tilings and the Harvey-Polking Lemma in $\G$}\label{section-3}
The first part of this section is simply restating the results from Section $3$ in \cite{CT}, so we direct the reader there for all the proofs. 

The purpose of this section is introducing a sort of "dyadic structure" in Carnot groups which is conducive to the removability results for Sections \ref{section-4} and \ref{section-6}. In particular, introducing a structure which allows for a generalization of the well-known Harvey-Polking partition of unity \cite[Lemma 3.1]{HP}. We use a modified Hausdorff measure, but the introduction of this measure requires some setup. Let $| \cdot |$ denote Lebesgue measure. 
\begin{thm}\cite[Theorem 3.1]{CT}
    Let $\G$ be a Carnot group of homogeneous dimension $A$. There exists $\frac{1}{2}$-homotheties $f_1, \cdots , f_M \in \G$, $M= 2^Q$, and a compact set $T \subset \G$ so that $T = \bigcup_{j=1}^M T_j$ where $T_j = f_j(T)$. Moreover $$0 < |T| < \infty$$ and $|T_j \cap T_i| = 0$ whenever $j \neq i$.
\end{thm}
We will call $T$ the \textit{fundamental tile}. Let $W=\{1,\ldots,M\}$. For $m\geq 0$ and $w=w_1\cdots w_m \in W^m$ we
introduce the notation
\begin{equation}
f_w = f_{w_1}\circ\cdots\circ f_{w_m} \nonumber
\end{equation}
and $T_w=f_w(T)$. We denote by $\mathcal{D}_m$ the family of sets $T_w$ as $w$
ranges over $W^m$.

\begin{prop}\cite[Proposition 3.4]{CT}\label{T-property-1}
The set $T$ is the closure of an open set.
\end{prop}

By Proposition \ref{T-property-1}, select a point $p \in
T$ and radii $0<\Rinn<\Rout$ so that $B(p,\Rinn) \subset T
\subset B(p,\Rout)$. Fixing $p$, we assume that $\Rinn$ has been
chosen as large as possible and $\Rout$ has been chosen as small as
possible subject to the preceding constraint. Then $\Rout < \diam T$. We call $p$ the
\textit{center} of $T$ and $\Rinn$ and $\Rout$ the \textit{inner} and
\textit{outer radii} of $T$, respectively.

For each $w \in W^m$, $m\ge 0$, we define the center of the \textit{tile}
$T_w$ to be $p_w = f_w(p)$, and the inner and outer radii of $T_w$ to be
$$
\Rinn_w = 2^{-m}\Rinn \quad \mbox{and} \quad \Rout_w=2^{-m}\Rout,
$$
respectively. We have $B(p_w,\Rinn_w) \subset T_w \subset B(p_w,\Rout_w)$.

The Hausdorff measures and dimensions of bounded subsets of $\G$ can
be computed using a ``dyadic Hausdorff measure'' constructed using the
tiles $\{T_w\}$. By applying a dilation, we may restrict our attention
to subsets of the initial tile $T$. We explicitly remark that we can do this since we will only be working with bounded subsets of $\G$. Let $\cD_* = \cup_{m\ge 0} \cD_m$.
For a set $A \subset T$ and $s,\varepsilon>0$, define
$$
\cH_{\cD,\varepsilon}^s(A) = \inf \sum_i (\diam T_{w_i})^s
$$
where the infimum is taken over all coverings of $A$ by tiles $T_{w_i}
\in \cD_*$ with $\diam T_{w_i} < \varepsilon$. Define $\cH_\cD^s(A) =
\lim_{\varepsilon \rightarrow 0} \cH_{\cD,\varepsilon}^s(A)$. For each $s>0$, $\cH_\cD^s$ is a
Borel measure on $T$. If the intersection of two
tiles has nonempty interior, then one of the tiles is contained inside
the other. As a result we can without loss of generality restrict our
attention in the definition of $\cH_{\cD}^s$ to essentially disjoint
coverings.

\begin{prop}\cite[Proposition 3.5]{CT}\label{T-property-2}
For each $s>0$ there exists a constant $C=C(s)>0$ so that
$\cH^s(A) \leq \cH^s_\cD(A) \leq C \cH^s(A)$ for every $A \subset T$.
\end{prop}
The following lemma is actually used to prove Proposition \ref{T-property-2}, but we will use it later.
\begin{lem}\label{doubling_spaces}
There exists a constant $N > 0$ so that for any ball $B(q,r)$ with $r \leq R^o$ and $m \geq 0$ chosen so that $2^{-m-1} R^o \leq r < 2^{-m} R^o$, it holds that the number of tiles $T_w \in \cD_m$ which intersect $B(q,r)$ is at most $N$.
\end{lem}
The above lemma was originally stated for $r \leq 1$ in \cite[Lemma 3.6]{CT}, but the same proof works when we assume $r \leq R^o$. The following is a generalization of the Harvey-Polking lemma in $\G$. The lemma uses the notation of (\ref{multi-index-notation}).
\begin{lem}\cite[Lemma 3.8]{CT}\label{T-property-7}
Let $\{T_{w_i}:1\leq i\leq N\}$ be a finite collection of essentially
disjoint tiles, with $T_{w_i} \in \cD_{m(i)}$. For each $i$ there is a
function $\varphi_i \in C^\infty_0(\G)$, supported in
$B(p_{w_i},2\Rout_{w_i})$, so that
$$
\sum_{i=1}^N \varphi_i(q) = 1 \quad \mbox{for all $q \in
  \bigcup_{i=1}^N T_{w_i}$.}
$$
Moreover, for each multi-index $\alpha$ there exists a constant
$C_\alpha>0$ so that
$$
|X_\alpha \varphi_i(x)| \le C_\alpha 2^{m(i)|\alpha|} \quad
\mbox{for all $x \in \G$ and $1\leq i\leq N$.}
$$
\end{lem}

By Proposition \ref{T-property-2}, we already know our new dyadic Hausdorff measure and the original Hausdorff measure are comparable, however we will also need to show that their Hausdorff contents are comparable.
\begin{lem}\label{comp_of_measures}
For a compact $K \subset T$ we have, $$\mathcal{H}_{\mathcal{D}, \infty}^s(K) \sim \mathcal{H}_{\infty}^s(K)$$
\end{lem}
\begin{proof}
In all cases, $\mathcal{H}_{\infty}^s(E) \leq \mathcal{H}_{\mathcal{D},\infty}^s(E)$ for any set $E \subset T$. So it suffices to show the other inequality. Let $\mathcal{S}_{\infty}^s$ denote the $s$-dimensional spherical Hausdorff content (with respect to the CC-distance). Then it suffices to show  $\mathcal{H}_{\mathcal{D},\infty}^s(K) \lesssim \mathcal{S}_{\infty}^s(K)$. Let $\varepsilon >0 $ be given and find a covering of $K$ by balls $\{ B(x_i, r_i)\}_{i=1}^M$ with $$\sum_i (2r_i)^s \leq \mathcal{S}_{\infty}^s(K) + \varepsilon.$$ First assume $r_i \leq R^o$ for all $i$. The other case is trivial and we show it at the end of the proof. By Lemma \ref{doubling_spaces} we can cover $B(x_i , r_i) \cap T$ by tiles $\{ T_{i,j} : 1\leq j \leq N\}$ where $N$ is independent of $i$ and $\diam T_{i,j} \sim 2r_i$. The claim that $N$ is independent from $i$ from follows $\G$ being a doubling metric space. See the proof of \cite[Lemma 3.6]{CT}. Since $K \subset T$, $K$ is covered by the tiles $\{ T_{i,j}\}_{i,j}$ with $\diam T_{i,j} \leq 2Cr_i$ for some fixed $C$. Let $r = \max_i r_i$. Then since $\diam T_{i,j} \leq 2Cr$ for all $i,j$, \begin{align*}
    \mathcal{H}_{\mathcal{D},\infty}^s(K) \leq \mathcal{H}_{\mathcal{D},C2r}^s(K) &\leq \sum_{i,j} (\diam T_{i,j})^s \\ &\leq C N \sum_{i}(2r_i)^s \\ &\leq C N(\mathcal{H}_{\infty}^s(K) +\varepsilon)
\end{align*}
and we're done in the case $r_i \leq R^o$ for all $i$. 

Now assume there exists $i$ such that $r_i \geq R^o$. Then obviously $T$ is a cover $K$ as we already assumed $K \subset T$. So since $R^o \leq \diam T \leq 2R^o$ we get, $$\mathcal{H}_{\mathcal{D}, \infty}^s(K) \leq (\diam T)^s \leq \sum_i (2r_i)^s \leq \mathcal{S}_{\infty}^s(K) + 
\varepsilon$$ which completes the proof.
\end{proof}

\vskip2em

\section{Capacity and Removability for Campanato Spaces}\label{section-4}
For Sections \ref{section-4} and \ref{section-6} we can without loss of generality assume that any compact $K$ we are working with satisfies $K \subset T$ by dilating and translating $K$.

Define the Campanato space with exponent $\varrho$ in $\Omega$, denoted $L_C^{\varrho}(\Omega)$, as functions $f \in L_{\text{loc}}^1(\Omega)$ satisfying, \begin{equation}\label{camp_def_1}
    \sup_{B(x,r) \subset \Omega} \frac{1}{r^{\varrho}} \int_{B(x,r)} \left| f(y) - f_{B(z,r)} \right|\, dy < \infty
\end{equation}
or 
\begin{equation}\label{camp_def_2}
\sup_{B(x,r)\subset \Omega} \inf_{c \in \mathbb{R}} \frac{1}{r^{\varrho}} \int_{B(x,r)} \left| f(y) - c \right|\, dy < \infty
\end{equation}
where $f_{B(z,r)} = r^{-Q} \int_{B(z,r)} f(x)\,dx$. 
For $f \in L_C^{\varrho}(\Omega)$, denote the quantity \ref{camp_def_1} as $\| f\|_*$. Furthermore, when $\Omega = \G$ we set $L_C^{\varrho}(\G) = L_C^{\varrho}.$

\begin{rem}
Notice that when $\varrho = Q$ and $\Omega = \G$, we get the space of functions of bounded mean oscillation (BMO). Hence, the definition of the Campanato space we have above is a slight generalization of BMO. 
\end{rem}

Now fix a differential operator $\L$. The following definition is basically given by Harvey and Polking in \cite[Definition 1.1]{HP2}.
\begin{defn}[Campanato Capacity]\label{camp-cap}
 The capacity of a compact $K \subset \G$ with respect to $L_C^{\varrho}$ and $\L$ is defined as, \begin{align*}
\k^{\varrho}(K) :=  \sup \left\{ |\langle \L f, 1 \rangle |: f \in L_C^{\varrho}, \| f\|_{*} \leq 1, \supp(\L f)\subset K \right\}.    
 \end{align*}  
 Above, $\supp(\L f)$ refers to the distributional support of $\L f$.
\end{defn}

The main result for this section we are after is the following:
\begin{thm}\label{main_result}
Let $\L$ be a left-invariant differential operator homogeneous of degree $\l \in [1,Q)$ such that both $\L$ and $\L^t$ are hypoelliptic. For a compact set $K \subset T$ and $\varrho \in [\l , Q]$ we have, $$\h^{\varrho-\lambda}(K) \lesssim \k^{\varrho}(K) \lesssim \h^{\varrho-\l}(K).$$    
\end{thm}
Before giving the proof, we will show that a certain function is in $L_C^{\varrho}$. For the rest of this section, we will take $\L$ to be a left-invariant differential operator homogeneous of degree $\l \in [1, Q)$. Since $\L$ and $\L^t$ are hypoelliptic, by Theorem \ref{folland-fund-soln}, there exists a kernel of type $\l$, denoted $\cg$, such that $\cg$ is  homogeneous of degree $\l-Q$ on $\G \backslash \{0\}$. Moreover, $\cg$ is a fundamental solution to $\L$. The proof below is essentially the same as one of the directions in \cite[Theorem 4.4]{CT}.

\begin{lem}\label{conv_is_in_L_C}
Fix a compact set $K$ and $\varrho \in [\l , Q]$. Let $\mu$ be a finite measure supported on $K$ such that $\mu(B(x,r)) \lesssim r^{\varrho - \l}$ for any ball. Then, $$f(x) :=  \mu(x) \ast \cg =  \int \cg(y^{-1} \cdot x)\,d\mu(y) \in L_C^{\varrho}.$$
\end{lem}

\begin{proof}
Proposition \ref{hom-func-bound} says $|\cg(p)| \lesssim \| p\|^{\l - Q}$ for all $p \in \G \backslash \{0\}$. Fix a ball $B(x_0 ,r)$ and set $\mu_1 := \mu \restrict_{B(x_0 ,2r)}$ and $\mu_2 := \mu \restrict_{B(x_0, 2r)^c}$. By Fubini and Proposition \ref{hom-func-bound}, \begin{align*}
    \left| \int_{B(x_0, r)} \int \cg(y^{-1} \cdot x) \, d\mu_1(y) \,dx \right| &\lesssim \int_{B(x_0 ,2r)} \int_{B(x_0, r)} \| y^{-1} \cdot x\|^{\l - Q} \,dx \,d\mu_1(y) \\ &\leq \int_{B(x_0 ,2r)} \left( \sum_{j=0}^{\infty} \int_{2^{-j-1}3r \leq \| y^{-1} \cdot x\| < 2^{-j}3r} \| y^{-1} \cdot x\|^{\l - Q} \,dx \right) \,d\mu_1(y) \\ &\lesssim \int_{B(x_0 ,2r)} \left( \sum_{j=0}^{\infty} (2^{-j-1} r)^{\l - Q} \int_{\| y ^{-1} \cdot x\| < 2^{-j}3r} \,dx \right) \,d\mu_1(y) \\ &\sim \int_{B(x_0 ,2r)} \left( \sum_{j=0}^{\infty} 2^{-j \l} r^{\l} \right) \, d\mu_1(y) \\ &\lesssim r^{\varrho}.
\end{align*}
Hence, \begin{equation}\label{estimate_on_mu_1}
\int_{B(x_0,r)} \left|  \mu_1(x) \ast \cg \right|\,dx \lesssim r^{\varrho}.    
\end{equation}
For $x \in B(x_0,r)$, $$\left|  \mu_2(x) \ast \cg -  \mu_2(x_0) \ast \cg \right| \leq \int_{B(x_0 ,2r)^c} \left| \cg(y^{-1} \cdot x) - \cg(y^{-1} \cdot x_0) \right| \, d\mu(y).$$ Then since $\cg$ is homogeneous of degree $\l - Q$ on $\G \backslash \{ 0\}$, by \cite[Proposition 3.1.40]{FR}\begin{equation}\label{pw_fund_soln_estiamte}
 \left| \cg( X \cdot Y) - \cg(X)\right| \lesssim \| Y\| \|X\|^{\l - Q -1} \quad \text{ for all }\| Y\| \leq \|X\|/2.   
\end{equation}
Setting $X = y^{-1} \cdot x_0$ and $Y = x_0^{-1} \cdot x$ we have $\| Y\| \leq r \leq \| X\| / 2.$ So (\ref{pw_fund_soln_estiamte}) gives, \begin{align*}
\left| \cg(y^{-1} \cdot x) - \cg(y^{-1} \cdot x_0) \right| &\lesssim \| x_0^{-1} \cdot x\| \| y^{-1} \cdot x_0\|^{\l- Q -1} \\ &\leq r \| y^{-1} \cdot x_0\|^{\l - Q -1}.
\end{align*}
Therefore for all $x \in B(x_0, r)$, \begin{align*}
\left|  \mu_2(x) \ast \cg - \mu_2(x_0) \ast \cg \right| &\lesssim r \int_{B(x_0 , 2r)^c} \| y^{-1} \cdot x_0\|^{\l -Q - 1} \,d\mu(y) \\ &= r \sum_{j=0}^{\infty} \int_{2^j 2r \leq \| y^{-1} \cdot x\| < 2^{j+1}2r} \| y^{-1} \cdot x_0\|^{\l - Q-1}\,d\mu(y) \\ &\lesssim r \sum_{j=0}^{\infty} 2^{j(\l - Q -1)} r^{\l - Q - 1} \mu(B(x_0, 2^{j+1} 2r)) \\ &\lesssim r^{\varrho - Q}.
\end{align*}
so that, \begin{equation}\label{int_estimate_on_fund}
 \int_{B(x_0,r)} | \cg \ast \mu_2(x) - \cg \ast \mu_2(x_0)|\,dx \lesssim r^{\varrho}.   
\end{equation}
Combining (\ref{estimate_on_mu_1}) and (\ref{int_estimate_on_fund}) we get, \begin{align*}
    \int_{B(x_0,r)} |  \mu(x) \ast \cg -  \mu_2(x_0) \ast \cg| \,dx &\leq \int_{B(x_0 , r)} | \mu_1(x) \ast \cg|\,dx + \int_{B(x_0 ,r)} |  \mu_2(x)\ast \cg -  \mu_2(x_0) \ast \cg|\,dx \\ &\lesssim r^{\varrho}.
\end{align*}
In view of (\ref{camp_def_2}), we have $f \in L_C^{\varrho}$.
\end{proof}

\begin{proof}[Proof of Theorem \ref{main_result}]
 First we show the lower bound. The inequality is trivial if $\mathcal{H}^{\varrho - \l}(K) = 0$, so assume $\mathcal{H}^{\varrho - \l}(K) > 0$. By Frostman's lemma there exists a measure $\mu$ supported on $K$ such that $\mu(B(x,r)) \lesssim r^{\varrho - \l}$ for all $x \in \G$ and $r > 0$, and such that $\mathcal{H}_{\infty}^{\varrho - \l}(K) \lesssim \mu(K)$. By Lemma \ref{comp_of_measures} we also have $\mathcal{H}_{\mathcal{D},\infty}^{\varrho-\l}(K) \lesssim \mu(K)$. Define,  
 \begin{align*}
 f(x) &:=  \mu(x) \ast \cg = \int \cg( y^{-1} \cdot x) \, d\mu(y) = \int_K \cg(y^{-1} \cdot x) \, d\mu(y).
 \end{align*} 
 Since $y^{-1} \cdot x = \tau_{y^{-1}} (x)$ we have $\L f(x) = 0$ for all $x \in K^c$, which follows by left-invariance of $\L$. So in the sense of distributions, $\supp (\L f) \subset K$. Furthermore, by Lemma \ref{conv_is_in_L_C} we have $f \in L_C^{\varrho}$. Take $\varphi \in C_0^{\infty}$ such that $\varphi \equiv 1$ in a neighborhood of $K$. Using $\L f \equiv 0$ on $K^c$, we can use Fubini to say, \begin{align}
    \k^{\varrho}(K) \geq  \left| \langle \L f, 1 \rangle \right|= | \langle \L f, \varphi \rangle | &= \left| \int \int \cg(y^{-1} \cdot x)\L^t\varphi(x) \, d\mu(y) \,dx \right| \nonumber \\ &= \left| \int \varphi(y) \,d\mu(y) \right| \nonumber \\ &\gtrsim \mathcal{H}_{\mathcal{D},\infty}^{\varrho -\l}(K) \nonumber
\end{align}
where we used $\cg$ being the fundamental solution of $\L$ to say $\int \cg(y^{-1}\cdot x)\L^t \varphi(x)\,dx = \varphi(y)$. This shows the lower bound. 

The upper bound is trivial if $\mathcal{H}_{\infty}^{\varrho -\l}(K) = \infty$, so assume $\mathcal{H}_{\infty}^{\varrho -\l}(K) < \infty$. Lemma \ref{comp_of_measures} says $\mathcal{H}_{\mathcal{D},\infty}^{\varrho -\l}(K) < \infty$ as well. Let $\varepsilon > 0$. By compactness we can find a finite collection of tiles $\{T_j\}_{j=1}^N$ such that 
$$K \subset \bigcup_{j=1}^N T_j \quad \text{ and } \quad \sum_{j=1}^N (\diam T_j)^{\varrho - \l} \leq \mathcal{H}_{\mathcal{D},\infty}^{\varrho-\l}(K) + \varepsilon.$$ 
We can also find balls $\{ B(p_j , 2R_j)\}_{j=1}^N$ such that $R_j \leq \diam T_j < 2 R_j$. By Lemma \ref{T-property-7} there exists function $\{ \varphi_j\}_{j=1}^N$ such that $\supp \varphi_j \subset B(p_j , 2R_j)$, $\sum_{j=1}^N \varphi_j \equiv 1$ on $\bigcup_j T_j$, and $\| X_{\a} \varphi_j \|_{\infty} \lesssim (\diam T_j)^{-|\a|}$ for $\a$ and $j$. Find $f \in L_C^{\varrho}$ such that $\supp (\L f) \subset K$ and \begin{equation}\label{upper_bound_on_cap} \k^{\varrho}(K) \leq |\langle \L f, 1 \rangle | + \varepsilon.\end{equation}
The Poincar\'{e}-Birkhoff-Witt Theorem says $\L$ is a linear combination of differential operators $X_{\a_{\ell}}$ with $|\a_{\ell}| = \l$. So we can assume without loss of generality that $\L = X_{\a}$ for $|\a| = \l$. Then $\L^t = (-1)^{\l} X_{\a}$. Recall the notation $X_{\a}^{\beta}$ as in (\ref{notation_for_X_B}) and the product formula (\ref{product_rule}) for $X_{\a}$. Set $c_j = f_{B(p_j , 2R_j)}$.
Then for $\psi \in C_0^{\infty}$ such that $\psi \equiv 1$ in a neighborhood of $K$ and using (\ref{length-of-beta-index}) , we can then write
\begin{align}
    | \langle \L f, 1 \rangle | &= |\langle \L f, \psi \rangle | \nonumber \\ &= \left| \langle \L f, \psi \sum_{j=1}^N \varphi_j \rangle \right| \nonumber \\ &= \left| \sum_{j=1}^N \langle \L (f-c_j) , \psi \varphi_j \rangle \right| \nonumber \\ &= \left| (-1)^{\l} \sum_{j=1}^N \langle f - c_j , \sum_{\beta \in \mathcal{B}} X_{\a}^{\beta}(\varphi_j) \star{X}_{\a}^{\beta}(\psi) \rangle \right| \nonumber \\ &\leq \sum_{j=1}^N |\langle f -c_j ,\psi X_{\a} \varphi_j \rangle| + \sum_{j=1}^N \left| \langle f -c_j , \sum_{\beta \in \mathcal{B}, \beta \neq \overline{\beta}} X_{\a}^{\beta}(\varphi_j)\star{X}_{\a}^{\beta}(\psi) \rangle \right| \nonumber \\ &\lesssim \sum_{j=1}^N \int_{B(p_j , 2R_j)} |f(y) - c_j| \| X_{\a} \varphi_j\|_{\infty}\,dy \nonumber \\ &\,\,\,\,\,\,\,\,\,\,\,\,\,\,\,\,\,\,\,\,\,\, + \sum_{j=1}^N \sum_{\beta \in \mathcal{B}, \beta\neq \overline{\beta}} \int_{B(p_j , 2R_j)} |f(y) - c_j| \| X_{\a}^{\beta} \varphi_j\|_{\infty}\,dy \nonumber \\ &\lesssim \sum_{j=1}^N (\diam T_j)^{-\l} (2R_j)^{\varrho} + \sum_{j=1}^N \sum_{\beta \in \mathcal{B} , \beta \neq \overline{\beta}} (\diam T_j)^{-\sum_{i=1}^{\l} \beta_i}(2R_j)^{\varrho} \nonumber \\ &\sim \sum_{j=1}^N\left( (\diam T_j)^{\varrho - \l} + \sum_{\beta \in \mathcal{B}, \beta \neq \overline{\beta}} (\diam T_j)^{\left(\sum_{i=1}^{\l} 1-\beta_i \right) + \varrho -\l}\right) \label{almost_final_inequality}
\end{align}
Now since $T$ is compact, there exists $J \geq 1$ such that $T \subset B(0, J/2)$ and therefore $\diam T \leq J$. Hence, $(\diam T_j) / J \leq 1$ for all $j$. Resuming from (\ref{almost_final_inequality}) we have, \begin{align*}
    |\langle \L f, 1 \rangle| &\lesssim \sum_{j=1}^N \left( J^{\varrho - \l} (\diam T_j / J)^{\varrho - \l} + \sum_{\beta \in \mathcal{B} , \beta \neq \overline{\beta}} J^{\left(\sum_{i=1}^{\l} 1-\beta_i \right) + \varrho -\l} (\diam T_j / J)^{\left(\sum_{i=1}^{\l} 1-\beta_i \right) + \varrho -\l} \right) \\ &\leq J^{\varrho} \sum_{j=1}^N \left( \diam T_j / J\right)^{\varrho - \l}  \\ &\leq J^{\l} (\mathcal{H}_{\mathcal{D},\infty}^{\varrho-\l}(K) + \varepsilon) \lesssim J^{\l}(\h^{\varrho-\l}(K) +\varepsilon)
\end{align*}
which, combined with (\ref{upper_bound_on_cap}), shows $$\k^{\varrho}(K) \lesssim J^{\l}(\h^{\varrho-\l}(K) +\varepsilon) + \varepsilon.$$ This completes the proof since $J$ only depends on the fundamental tile $T$.   
\end{proof}

Characterizing removability through $\k^{\varrho}$ will follow from Theorem \ref{main_result}. We would like to be able to immediately cite \cite[Theorem 4.4]{CT} and say removability follows as a corollary to Theorem \ref{main_result}, however their definition of removability is slightly different than what we are using here. In particular, they require the condition $K \subset \Omega$ in addition to what we stated in Definition \ref{removable-sets-defn} (also note that \cite[Theorem 4.4]{CT} characterizes removable sets for BMO functions, not Campanato functions, but the same proof they use goes through for $L_C^{\varrho}$, $\varrho \in (\l, Q]$). So we will reprove one of the directions of \cite[Theorem 4.4]{CT} under the definition of removability in this paper.

\begin{thm}[Removability Characterization for $L_C^{\varrho}$]\label{rem-iff-hausdorff-zero}
 Fix $\l \in [1,Q)$ and let $\L$ be $\l$-homogeneous, left-invariant linear differential operator on $\G$ such that both $\L$ and $\L^t$ are hypoelliptic. Let $\varrho \in [\l, Q]$. Then a compact set $K \subset T$ is removable for $L_C^{\varrho}$ $\L$-solutions if and only if $\k^{\varrho}(K) = 0$.   
\end{thm}
\begin{proof}
First assume $K$ is removable. Take any $f \in L_C^{\varrho}$ as in the definition of $\k^{\varrho}(K)$. In particular, $\supp (\L f) \subset K$. This means $\L f \equiv 0$ (as a function) on $K^c$. Applying Definition \ref{removable-sets-defn} which $\Omega = \G$ gives $\L f = 0$ on $\G$. Since $f$ was arbitrary this implies $\k^{\varrho}(K) = 0$.

Now assume $\k^{\varrho}(K) = 0$. By Theorem \ref{main_result} and Proposition \ref{T-property-2}, $\mathcal{H}_{\mathcal{D}}^{\varrho-\l}(K) = 0$. Take any domain $\Omega$ and $f \in L_C^{\varrho}(\Omega)$ for which $\L f = 0$ in $\Omega \backslash K$. Assume for contradiction that $\L f \neq 0$ in $\Omega$. Clearly we can assume $K \cap \Omega \neq \emptyset$. Find a ball $B(x,r)$ which for which $\overline{B(x,r)} \subset \Omega$ and $\L f \neq 0$ in $\overline{B(x,r)}$. We can do this by the assumption $\L f \neq 0$ in $\Omega$ and since $\Omega$ is open. Set $F = \overline{B(x,r)}$, a compact set. Observe that $K \cap F \neq \emptyset$. Indeed if $K \cap F = \emptyset$ then along with $F \subset \Omega$, this implies $\Omega \backslash K \supset F$ and hence, $\L f =0$ in $F$. This is a contradiction. So $K \cap F \neq \emptyset$. Furthermore, $\mathcal{H}_{\mathcal{D}}^{\varrho -\l}(K \cap F) = 0$. Set $d_0 := d(K \cap F , \Omega^c)$. Take any $\varepsilon > 0$ such that $\varepsilon < \min\{1, (d_0/4)^{\varrho-\l}\}$. By compactness, there exists a finite collection of tiles $\{ T_{w_j}\}_{{j=1}}^N$ such that $K \cap F \subset \bigcup_j T_{w_j}$ and \begin{align*}
    \sum_{j=1}^N (\diam T_{w_j})^{\varrho -\l} < \varepsilon.
\end{align*}
By shrinking $\varepsilon$ if necessary, we note that $\bigcup_j B(p_{w_j}, 2\Rout_{w_j}) \subset \Omega$ since $\Rout_{w_j} \leq \diam T_{w_j} \leq 2\Rout_{w_j}$.

By Lemma \ref{T-property-7} there exists a family of functions $\{ \varphi_j \}_{j=1}^N$ such that $\supp \varphi_j \subset B(p_{w_j}, 2\Rout_{w_j})$, $\sum_{j=1}^N \varphi_{j} \equiv 1$ on $\bigcup_j T_{w_j}$, and $\| X_{\a} \varphi_{j}\|_{\infty} \lesssim (\diam T_{w_j})^{-|\a|}$ (we are using the multi-index notation of (\ref{multi-index-notation}) here). Since $\L f \neq 0$ in $F$ we can find $\psi \in C_0^{\infty}(F)$ for which $| \langle \L f, \psi \rangle | > 0$ (and $\psi$ does not depend on $\varepsilon$). In the sense of distributions, since $F \subset \Omega$ we have, \begin{align}
 \supp (\psi \L f) &\subset F \cap (\Omega \backslash K)^c \nonumber =
 F \cap K. \nonumber
\end{align} The Poincar\'{e}-Birhkoff-Witt Theorem says $\L$ is a linear combination of differential operators $X_{\a_{\ell}}$ with $| \a_{\ell}| = \l$. So we can assume without loss of generality that $\L =X_{\a}$ for $|\a| = \l$. Then $\L^t  = (-1)^{\l}X_{\a}$. Recall the notation $X_{\a}^{\beta}$ as in (\ref{notation_for_X_B}) and the product formula (\ref{product_rule}) for $X_{\a}$. Therefore, since $\{ \varphi_j\}_j$ is a partition of unity on $ F \cap K \supset \supp (\psi \L f)$ we have with $c_j = f_{B(p_{w_j},2\Rout_{w_j})}$, \begin{align*}
|\langle \L f, \psi \rangle| = |\langle \psi \L f , 1 \rangle| &= \left| \langle \psi \L f , \sum_{j=1}^N \varphi_j \rangle \right| \\ &= \left| \sum_{j=1}^N \langle \L (f-c_j) , \psi \varphi_j \rangle \right| \\ &= \left| (-1)^{\l} \sum_{j=1}^N \langle f - c_j , \sum_{\beta \in \mathcal{B}} X_{\a}^{\beta}(\varphi_j) \star{X}_{\a}^{\beta}(\psi) \rangle \right|  \\ &\leq \sum_{j=1}^N |\langle f -c_j ,\psi X_{\a} \varphi_j \rangle| + \sum_{j=1}^N \left| \langle f -c_j , \sum_{\beta \in \mathcal{B}, \beta \neq \overline{\beta}} X_{\a}^{\beta}(\varphi_j)\star{X}_{\a}^{\beta}(\psi) \rangle \right| \\ &\lesssim \sum_{j=1}^N \int_{B(p_j , 2\Rout_{w_j})} |f(y) - c_j| \| X_{\a} \varphi_j\|_{\infty}\,dy \\ &\,\,\,\,\,\,\,\,\,\,\,\,\,\,\,\,\,\,\,\,\,\, + \sum_{j=1}^N \sum_{\beta \in \mathcal{B}, \beta\neq \overline{\beta}} \int_{B(p_j , 2\Rout_{w_j})} |f(y) - c_j| \| X_{\a}^{\beta} \varphi_j\|_{\infty}\,dy  \\ &\lesssim \sum_{j=1}^N (\diam T_{w_j})^{-\l} (2\Rout_{w_j})^{\varrho} + \sum_{j=1}^N \sum_{\beta \in \mathcal{B} , \beta \neq \overline{\beta}} (\diam T_{w_j})^{\left(\sum_{i=1}^{\l} 1-\beta_i \right)-\l}(2\Rout_{w_j})^{\varrho} \\ &\sim \sum_{j=1}^N\left( (\diam T_{w_j})^{\varrho - \l} + \sum_{\beta \in \mathcal{B}, \beta \neq \overline{\beta}} (\diam T_{w_j})^{\left(\sum_{i=1}^{\l} 1-\beta_i \right) + \varrho -\l}\right) \\ &\lesssim \sum_{j=1}^N (\diam T_{w_j})^{\varrho -\l} < \varepsilon.
\end{align*}
Since $\varepsilon > 0$ is arbitrary this implies $|
\langle \L f, \psi \rangle| = 0$, a contradiction. Hence, we indeed have $\L f = 0$ in $\Omega$, implying $K$ is removable.
\end{proof}

\vskip2em

\section{Capacity and Removability for H\"{o}lder Continuous Functions}\label{section-6}
Most of the proofs in this section will be almost exactly the same as in the previous section, so we will be light on details here.
\begin{defn}
Let $\Omega \subset \G$ open and $\delta \in (0,1)$. A function $f : \Omega \rightarrow \R$ belongs to $\text{Lip}_{\delta}(\Omega)$ if there exists a constant $C>0$ such that $$|f(x) - f(y)| \leq C\, d(x,y)^{\delta} \quad \text{ for all }x,y \in \Omega.$$ 
 \end{defn}
Let $\| \cdot \|_{\delta}$ denote the norm with respect to $\text{Lip}_{\delta}$. Fix a differential operator $\L$.

\begin{defn}[H\"{o}lder Capacity]
The capacity of a compact $K \subset \G$ with respect to $\text{Lip}_{\delta}$ and $\L$ is defined as, \begin{align*}
    \k^{\delta}(K) := \sup \{ | \langle \L f, 1 \rangle |: f \in \text{Lip}_{\delta}, \| f\|_{\delta} \leq 1, \supp(\L f) \subset K\}.
\end{align*}    
\end{defn} 
The main result of this section is the following:
\begin{thm}\label{main-result-holder}
Let $\L$ be a left-invariant differential operator homogeneous of degree $\l \in [1, Q)$ such that both $\L$ and $\L^t$ are hypoelliptic. For a compact $K \subset T$ and $\delta \in (0,1)$ we have, \begin{align*}
    \h^{Q-\l + \delta}(K) \lesssim \k^{\delta}(K) \lesssim \h^{Q-\l+\delta}(K).
\end{align*} 
\end{thm}

\begin{rem}
Combining Theorems \ref{main_result} and \ref{main-result-holder} gives another characterization of Hausdorff dimension for compact sets $K$. Take $\L$ satisfying the above conditions and for which $\L$ homogeneous of degree $1$. By translating and dilating $K$ to be a subset of $T$, we can then consider when $\h^{\beta}(K)=0$ for $\beta \in [0, Q)$ and find either the Campanato or H\"{o}lder capacity which fits this $\beta$. The resulting capacity will then also be null. So if one suspects that the dimension of a set is strictly less than $Q$, then looking at the null sets for this capacity can give another way of finding the dimension of that set.
\end{rem}

As before, we will first show $\cg \ast \mu \in \text{Lip}_{\delta}$ for a Frostman measure $\mu$ and $\cg$ as in Theorem \ref{folland-fund-soln}. We provide the details of the proof of the following lemma. This proof is taken from \cite[Theorem 4.12]{CT}.
\begin{lem}
Fix a compact $K$ and $\delta \in (0,1)$. Let $\mu$ be a finite measure supported on $K$ such that $\mu(B(x,r)) \lesssim r^{\varrho - \l}$ for any ball. Then, \begin{align*}
    f(x) :=  \mu(x) \ast \cg = \int \cg(y^{-1}\cdot x)\, d\mu(y) \in \text{Lip}_{\delta}.
\end{align*}
\end{lem}
\begin{proof}
    Let $x,z \in \G$. Then, \begin{align*}
         |f(x) -f(z)| &\leq \int_{ \{ \| y^{-1} \cdot z\| > 2\|z^{-1} \cdot x\|\}} |\cg(y^{-1} \cdot x) - \cg(y^{-1} \cdot z)|\,d\mu(y) \\ &\,\,\,\,\,\,\,\,\,\,\,\,\,\,\,\,\,\,\,\,+ \int_{ \{ \| y^{-1} \cdot z\| \leq 2\|z^{-1} \cdot x\|\}} |\cg(y^{-1} \cdot x) - \cg(y^{-1} \cdot z)|\,d\mu(y) \\ &=: I_1 + I_2.
     \end{align*}
Since $\cg$ is homogeneous of degree $\l-Q$ on $\G \backslash \{0\}$, by \cite[Proposition 3.1.40]{FR} \begin{equation}\label{ineq-for-holder}
|\cg(X \cdot Y) - \cg(X)| \lesssim \| Y\| \| X\|^{\l-Q-1} \quad \text{ for all }\| Y\| \leq \|X\|/2.
\end{equation}
Then with $X = y^{-1} \cdot z$ and $Y = z^{-1} \cdot x$, equation (\ref{ineq-for-holder}) gives \begin{align*}
    I_1 &\lesssim \| z^{-1} \cdot x\| \int_{\{ \| y^{-1} \cdot z\| > 2\|z^{-1} \cdot x\|\}} \|y^{-1} \cdot z\|^{\l - Q -1}\,d\mu(y) \\ &\leq \| z^{-1} \cdot x\| \sum_{j=1}^{\infty} \int_{2^j \|z^{-1} \cdot x\| < \| y^{-1} \cdot z\| \leq 2^{j+1}\|z^{-1} \cdot x\|} \| y^{-1} \cdot z\|^{\l - Q-1}\,d\mu(y) \\ &\leq \| z^{-1} \cdot x\| \sum_{j=1}^{\infty} (2^j \| z^{-1} \cdot x\|)^{\l - Q -1} \mu(B(z, 2^{j+1}\|z^{-1} \cdot x\|)) \\ &\sim \| z^{-1} \cdot x\|^{\delta} \sum_{j=1}^{\infty} 2^{j(\delta - 1)} \\ &\sim \|z^{-1} \cdot x\|^{\delta}.
\end{align*}
Note that, $$\{ y \in \G : \| y^{-1} \cdot z\| \leq 2\|z^{-1} \cdot x\|\} \subset \{y \in \G : \| y^{-1} \cdot x\| \leq 3\| z^{-1} \cdot x\|\}.$$ Hence, \begin{align*}
    I_2 &\leq \int_{\|y^{-1} \cdot x\| \leq 3\|z^{-1} \cdot x\|} |\cg(y^{-1} \cdot x)|\,d\mu(y) + \int_{\| y^{-1} \cdot z\| \leq 2\|z^{-1} \cdot x\|} |\cg(y^{-1} \cdot z)|\,d\mu(y) \\ &=: I_{2_a} + I_{2_b}.
\end{align*}
Then by Proposition \ref{hom-func-bound}, \begin{align*}
    I_{2_a} &\lesssim \sum_{j=0}^{\infty}\int_{ 2^{-j-1}3\|z^{-1} \cdot x\| <\| y^{-1} \cdot x\| \leq 2^{-j} 3\|z^{-1} \cdot x\|} \| y^{-1} \cdot x\|^{\l - Q}\,d\mu(y) \\ &\lesssim \sum_{j=0}^{\infty} (2^{-j} 3 \| z^{-1} \cdot x\|)^{\l - Q} \mu(B(z, 2^{-j} 3 \| z^{-1} \cdot x\|)) \\ &\lesssim \| z^{-1} \cdot x\|^{\delta} \sum_{j=0}^{\infty} 2^{-j \delta} \\ &\sim \| z^{-1} \cdot x\|^{\delta}
\end{align*}
and similarly, \begin{align*}
   I_{2_b} =\int_{\| y^{-1} \cdot z\| \leq 2\|z^{-1} \cdot x\|} |\cg(y^{-1} \cdot z)|\,d\mu(y) \lesssim \|z^{-1} \cdot x\|^{\delta}.
\end{align*}
Therefore, $|f(x) - f(z)| \lesssim d(x,z)^{\delta}$ which completes the proof.
\end{proof}
We can now prove the main result of this section. We will not provide details though.
\begin{proof}[Proof of Theorem \ref{main-result-holder}]
The lower bound follows in the same way as the lower bound for Theorem \ref{main_result}. The upper bound uses almost exactly the same proof as the upper bound in Theorem \ref{main_result}, except we take $c_j = f(p_{w_j})$.
\end{proof}
Again, we can use Theorem \ref{main-result-holder} to characterize removable sets for H\"{o}lder continuous functions. 

\begin{thm}[Removability Characterization for $\text{Lip}_{\delta}$]
Fix $\l \in [1,Q)$ and let $\L$ be $\l$-homogeneous, left-invariant linear differential operator on $\G$ such that both $\L$ and $\L^t$ are hypoelliptic. Let $\delta \in (0,1)$. Then a compact set $K \subset T$ is removable for $\text{Lip}_{\delta}$ $\L$-solutions if and only if $\k^{\delta}(K) = 0$.     
\end{thm}
\begin{proof}
As before, the assumption that $K$ is removable immediately gives $\k^{\delta}(K) = 0$. For the other direction when $\k^{\delta}(K) = 0$, the proof is the same as in Theorem \ref{rem-iff-hausdorff-zero} except we take $c_j = f(p_{w_j})$.
\end{proof}

\vskip2em

\section{Capacity and Removability for $L^p_{\text{loc}}$}\label{section-7}
To characterize removable sets for $L^p_{\text{loc}}$, we will not follow the same pattern as in the previous two sections. Rather, we can characterize removability in a fairly direct fashion as in Section \ref{lip-section}. Let $\| f\|_{L^p(A)}$ denote the $L^p$ norm on a Lebesgue measurable set $A \subset \G$.
\begin{defn}[$L^p_{\text{loc}}$ Capacity]
The capacity of a compact $K \subset \G$ with respect to $L^p_{\text{loc}}$ and $\L$ is defined as, \begin{align*}
    \k^{p}(K) := \sup \{ | \langle \L f, 1 \rangle |: f \in L^p_{\text{loc}}, \| f\|_{L^p(K)} \leq 1, \supp(\L f) \subset K\}.
\end{align*}    
\end{defn} 
In this section, we will use a homogeneity property of the product formula given in (\ref{product_rule}). As before, for a differential operator $\L$ satisfying the hypotheses of Theorem \ref{folland-fund-soln}, we get the existence of $\cg \in C^{\infty}(\G \backslash \{0\})$ which is homogeneous of degree $\l-Q$ on $\G \backslash \{0\}$ and $\L \cg = \delta$. It is easy to check that \begin{equation}\label{cg-in-l1}
X \cg \in L^1_{\text{loc}}
\end{equation} whenever $X$ is a differential operator homogeneous of degree \textbf{strictly less} than $\l$. One can integrate over annuli or use the polar coordinate formula of \cite[Proposition 5.4.4]{BLU} to see this. So the main purpose of the following product formula will be making sure that all relevant derivatives that hit $\cg$ satisfy this homogeneity property.

 Fix a left-invariant differential operator $\L$ homogeneous of degree $\l \in [1, Q)$. We will now find a distributional formula for $\psi \L f \ast \cg$ similar to what was done in Section \ref{lip-section}. The Poincar\'{e}-Birhkoff-Witt Theorem says $\L$ is a linear combination of multi-indices $X_{\a}$ with $| \a| = \l$, where again, we are using the multi-index notation of (\ref{multi-index-notation}). Hence, we can write $\L = \sum_{\a} c_{\a}X_{\a}$ for some finite index of $\a$'s. So for $\psi ,\varphi \in C^{\infty}$ we have, \begin{align}
\L (\psi \varphi) &= \sum_{\a} c_{\a} X_{\a} (\psi \varphi) \nonumber \\ &= \sum_{\a} c_{\a} (X_{\a} \psi )\varphi + \sum_{\a} \sum_{\beta \in \mathcal{B}, \beta \neq \overline{\beta}} c_{\a} X_{\a}^{\beta}(\psi) \star{X}_{\a}^{\beta}(\varphi) \nonumber \\ &= \varphi (\L \psi) + \sum_{\beta \in \mathcal{B}, \beta \neq \overline{\beta}} \sum_{\a} c_{\a} X_{\a}^{\beta}(\psi) \star{X}_{\a}^{\beta}(\varphi). \label{almost_product_rule}
\end{align}
Notice that for $\beta \neq \overline{\beta}$, $X_{\a}^{\beta}$ is a homogeneous operator of degree at most $\l-1$. To ease notation, rewrite the second sum of (\ref{almost_product_rule}) as, \begin{align*}
\sum_{\beta \in \mathcal{B}, \beta \neq \overline{\beta}} c_{\a} X_{\a}^{\beta}(\psi) \star{X}_{\a}^{\beta}(\varphi) &= \sum_{\eta} P_{\eta}(\psi) Q_{\eta}(\varphi)
\end{align*}
where $P_{\eta}$, $Q_{\eta}$ are differential operators and $\eta$ runs through some finite indexing set. Now we will come up a distributional formula in the style of Harvey-Polking for objects of the form $\psi\L f$. Take any open $\Omega \subset \G$, $\L f \in \D'(\Omega)$, and $\psi \in C_0^{\infty}(\Omega)$. Take any $\phi \in C_0^{\infty}(\Omega)$. Then, \begin{align*}
    \langle \psi \L f, \phi \rangle &= \langle f, \psi \L^t \phi \rangle + \sum_{\eta} \langle f, P_{\eta}^t(\phi) Q_{\eta}^t(\psi)\rangle \\ &= \langle \L (\psi f) , \phi \rangle + \sum_{\eta} \langle P_{\eta}(f Q_{\eta}^t(\psi)), \phi \rangle
\end{align*}
which means we have the distributional formula of,
\begin{equation}\label{distributional_product_formula}
 \psi \L f = \L (\psi f) + \sum_{\eta} P_{\eta}(f Q_{\eta}^t(\psi)).   
\end{equation}
By \cite[Corollary 2.8]{Fo}, $\L(\psi f) \ast \cg = \psi f$. So convolving $\cg$ with (\ref{distributional_product_formula}) gives, \begin{equation}\label{almost-distr-form-with-cg}
\psi \L f \ast \cg = \psi f + \sum_{\eta} P_{\eta}(f Q_{\eta}^t(\psi)) \ast \cg    
\end{equation}
Now notice that by (\ref{almost_product_rule}), $P_{\eta}$ is a composition of left-invariant vector fields. So by iterating the formula on \cite[Page 31]{FR} we can rewrite (\ref{almost-distr-form-with-cg}) as, \begin{equation}\label{distr-form-cg}
 \psi \L f \ast \cg = \psi f + \sum_{\eta} fQ_{\eta}^t \psi \ast \wt{P_{\eta}}\cg
\end{equation}
where $\wt{P_{\eta}}$ is a composition of right-invariant vector fields. Furthermore, it is straightforward to check that $\wt{P_{\eta}}$ has the same homogeneity degree as $P_{\eta}$. In particular, $\wt{P_{\eta}}$ is homogeneous of degree at most $\l-1$. 

\begin{thm}[Removability Characterization for $L^p_{\text{loc}}$]
Let $p \in [1, \infty)$. Fix $\l \in [1,Q)$ and let $\L$ be $\l$-homogeneous, left-invariant linear differential operator on $\G$ such that both $\L$ and $\L^t$ are hypoelliptic. A compact $K \subset \G$ is removable for $L^p_{\text{loc}}$ $\L$-solutions if and only if $\k^p(K) = 0$.
\end{thm}
\begin{proof}
As before, the direction of removable implying $\k^p(K) = 0$ is trivial. So assume $\k^p(K) = 0$. Take any $f \in L^p_{\text{loc}}$ with $\supp(\L f)\subset K$. Take any $\psi \in C_0^{\infty}$. Since $\cg$ is the fundamental solution to $\L$, left-invariance of $\L$ implies \begin{align}
|\langle \L f, \psi \rangle| &= |\langle \psi \L f, 1 \rangle | \nonumber \\ &= \left| \langle \L (\psi \L f \ast \cg), 1 \rangle \right|. \nonumber
\end{align}
It suffices to show $\psi \L f \ast \cg \in L^p_{\text{loc}}$. Indeed, first notice $\supp( \L (\psi \L f \ast \cg)) = \supp(\psi \L f) \subset K$. Hence, for $A := \| \psi \L f \ast \cg\|_{L^p(K)}$, if we can show $\psi \L f \ast \cg \in L_{\text{loc}}^p$ then the assumption $\k^p(K) = 0$ implies \begin{align*}
     |\langle \L (\psi \L f  \ast \cg), 1 \rangle |/A = |\langle \L f, \psi \rangle| /A = 0
\end{align*}
and we would be done. Recall that (\ref{distr-form-cg}) says we have the distributional relation of, \begin{equation}\label{show-each-in-lp}
\psi \L f \ast \cg = \psi f + \sum_{\eta} f Q_{\eta}^t \psi \ast \wt{P_{\eta}} \cg 
\end{equation}
where $Q_{\eta}^t$ and $\wt{P_{\eta}}$ have the properties previously discussed. We will show each function on the right hand side of (\ref{show-each-in-lp}) is in $L^p_{\text{loc}}$. Clearly $\psi f \in L^p_{\text{loc}}$. Now fix $\eta$ and consider $f Q_{\eta}^t\psi \ast \wt{P_{\eta}}\cg$. By (\ref{cg-in-l1}), $\wt{P_{\eta}} \cg \in L^1_{\text{loc}}$, and clearly $f Q_{\eta}^t \psi \in L^p_{\text{loc}}$. So Young's inequality says $f Q_{\eta}^t \psi \ast \wt{P_{\eta}} \cg \in L^p_{\text{loc}}$, which completes the proof.
\end{proof}

\printbibliography

@book{BLU,
  author = {A. Bonfiglioli and E. Lanconelli and F. Uguzzoni},
  year = {2007},
  title = {Stratified Lie groups and potential theory for their sub-Laplacians},
  publisher = {Springer Monographs in Mathematics},
  address = {Berlin},
  edition = {}
}

@article{Fo,
  title={Subelliptic estimates and function spaces on nilpotent Lie groups},
  author={G. B. Folland},
  journal={Ark. Mat.},
  volume={13},
  number={2},
  pages={161--207},
  year={1975},
  publisher={}
}

@book{FR,
  author = {V. Fisher and M. Ruzhanksy},
  year = {2016},
  title = {Quantization on nilpotent Lip groups},
  publisher = {Birkh\"{a}user Cham},
  series = {Progress in Mathematics},
  volume = {314},
  address = {},
  edition = {}
}

@phdthesis{Vi,
    title    = {Submanifolds in Carnot groups},
    school   = {Scuolo Normale Supierore},
    author   = {D. Vittone},
    year     = {2006}
}

@article{HP,
  title={Removable singularities of solutions of linear partial differential equations},
  author={R. Harvey and J. Polking},
  journal={Acta. Math},
  volume={125},
  number={},
  pages={39--56},
  year={1970},
  publisher={}
}

@article{CT,
  title={Removable sets for homogeneous linear PDE in Carnot groups},
  author={V. Chousionis and J. Tyson},
  journal={J. Analyse Math},
  volume={27},
  number={},
  pages={216--238},
  year={2016},
  publisher={}
}

@article{CMT,
  title={Removable sets for Lipschitz harmonic functions on Carnot groups},
  author={V. Chousionis and V. Magnani and J. Tyson},
  journal={Calc. Var. PDE},
  volume={3},
  number={3-4},
  pages={755--780},
  year={2015},
  publisher={}
}

@article{P,
  title={On harmonic approximation in the $C1$ norm},
  author={P.V. Paramonov},
  journal={Math. USSR-Sb},
  volume={71},
  number={},
  pages={183--207},
  year={1992},
  publisher={}
}

@article{MP,
  title={On geometric properties of harmonic $Lip_1$-capacity},
  author={P. Mattila and P.V. Parmonov},
  journal={Pacific J. Math},
  volume={171},
  number={2},
  pages={469--491},
  year={1995},
  publisher={}
}

@article{To,
  title={Painlev\'{e}'s problem and the semiadditivity of analytic capacity},
  author={X. Tolsa},
  journal={Acta Math},
  volume={190},
  number={1},
  pages={105--149},
  year={2003},
  publisher={}
}

@article{LDPS,
  title={Universal differentiability sets and maximal directional derivatives in Carnot groups},
  author={E. Le Donne and A. Pinamonti and G. Speight},
  journal={J. Math. Pures Appl},
  volume={121},
  number={9},
  pages={83--112},
  year={2019},
  publisher={}
}

@article{Ah,
  title={Bounded analytic functions},
  author={L.V. Ahlfors},
  journal={Duke Math. J.},
  volume={14},
  number={},
  pages={1--11},
  year={1947},
  publisher={}
}

@article{Car,
  title={Removable singularities of continuous harmonic functions in $R^m$},
  author={L. Carleson},
  journal={Math. Scand.},
  volume={12},
  number={},
  pages={15--18},
  year={1963},
  publisher={}
}

@article{Kr,
  title={Singularit\'{e}s non essentielles des solutions des\'{e}quations aux d\'{e}riv\'{e}es partielles.},
  author={J. Kr\'{a}l},
  journal={S\'{e}minaire de Th\'{e}orie du Potentiel, Paris, 1972-1974, Lecture Notes in Math.},
  volume={518},
  number={},
  pages={95--106},
  year={1976},
  publisher={Springer-Verlag}
}

@article{HP2,
  title={A notion of capacity which characterizes removable singularities},
  author={R. Harvey and J. Polking},
  journal={Trans. Amer. Math. Soc.},
  volume={169},
  number={},
  pages={183--195},
  year={1972},
  publisher={}
}

@article{Toref,
  title={Analytic capacity, rectifiability, and the Cauchy integral},
  author={X. Tolsa},
  journal={International Congress of Mathematicians Vol II},
  volume={},
  number={},
  pages={1505--1527},
  year={2006},
  publisher={Eur. Math. Soc.}
}

@article{Pan,
  title={M\'{e}triques de Carnot-Carath\'{e}odory et Quasiisom\'{e}tries des Espaces Sym\'{e}triques de rang un},
  author={P. Pansu},
  journal={Ann. of Math.},
  volume={129},
  number={1},
  pages={1--60},
  year={1989},
  publisher={}
}

@article{Ver,
  title={BMO Rational Approximation and One-Dimensional Hausdorff Content},
  author={J. Verdera},
  journal={Trans. Amer. Math. Soc.},
  volume={297},
  number={1},
  pages={283--304},
  year={1986},
  publisher={}
}

@article{GN,
  title={Lipschitz continuity, global smooth approximations and extension theorems for Sobolev functions in Carnot-Carath\'{e}odory spaces},
  author={N. Garofalo and D. M. Nhieu},
  journal={J. Anal. Math.},
  volume={74},
  number={},
  pages={67--97},
  year={1998},
  publisher={}
}

@article{CG,
  title={Ahlfors type estimates for perimeter measures in Carnot-Carath\'{e}odory spaces},
  author={L. Capogna and N. Garofalo},
  journal={J. Geom. Anal.},
  volume={16},
  number={},
  pages={455--497},
  year={2006},
  publisher={}
}

@article{M,
  title={A New Differentiation, Shape of the Unit Ball and Perimeter Measure},
  author={V. Magnani},
  journal={Ind. Univ. Math. J.},
  volume={66},
  number={1},
  pages={183--204},
  year={2017},
  publisher={}
}

@article{CMat,
  title={Singular integrals on self-similar sets and removability for Lipschitz harmonic functions in Heisenberg groups.},
  author={V. Chousionis and P. Mattila},
  journal={J. Reine Angew. Math.},
  volume={691},
  number={},
  pages={29--60},
  year={2014},
  publisher={}
}

@article{CFO,
  title={Boundedness of singular integrals on $C^{1,\alpha}$ intrinsic graphs in the Heisenberg group},
  author={V. Chousionis and K. F\"{a}ssler and T. Orponen},
  journal={Adv. Math.},
  volume={354},
  number={},
  pages={106745},
  year={2019},
  publisher={}
}

@article{CLZ,
  title={Singular integrals on $C^{1,\alpha}$ intrinsic graphs in step 2 Carnot groups},
  author={V. Chousionis and S. Li and L. Zhang},
  journal={J. Geom. Anal},
  volume={35},
  number={380},
  pages={},
  year={2025},
  publisher={}
}

\end{document}